\documentclass[preprint]{elsarticle}
\usepackage{cmap} 
           \usepackage[T2A]{fontenc}   
           \usepackage[utf8]{inputenc}
           \usepackage[russian,english]{babel}
\usepackage{amsmath}
\usepackage{amsfonts}
\usepackage{amssymb}
\usepackage{amsthm}
\usepackage[pdftex,unicode]{hyperref}
\usepackage{indentfirst} 
\usepackage{tikz-cd}

  \def\R{\mathbb R}

\def\om{\omega}
\def\Om{\Omega}

\def\be{\beta}

\def\al{\alpha}
\def\ga{\gamma}
\def\la{\lambda}
\def\ph{\varphi}
\def\de{\delta}
\def\ka{\kappa}

\def\nfs/{NFS}
\def\cdp/{CDP}
\def\cdpz/{CDP${}_0$}

\def\cB{\mathcal{B}}
\def\cE{\mathcal{E}}
\def\cW{\mathcal{W}}

\def\cH{\mathcal{H}}

\def\cl#1{\overline{#1}}

\def\Int{\operatorname{Int}}

\def\sp{\operatorname{sp}}

\def\es{\varnothing}
\def\Tau{{\mathcal T}}

\def\sset#1{\{#1\}}

\def\set#1{\bbset#1\eeset}
\def\bbset#1:#2\eeset{\{#1\,:\,#2\}}

\def\bbsett#1:#2\eesett{\{#1\,:\,\text{#2}\}}

\def\ibbset#1:#2\ieeset{(#1)_{#2}}

\def\cM{{\mathcal M}}

\def\tp{\Tau}

\def\cB{{\mathcal B}}

\def\cB{{\mathcal B}}

\def\cS{{\mathcal S}}
\def\cU{{\mathcal U}}
\def\cO{{\mathcal O}}

\def\cD{{\mathcal D}}

\def\wt#1{\widetilde{#1}}

\def\cI{{\mathcal I}}

\def\diag{\mathop{\bigtriangleup}}

\def\fun{{}^\frown}

\newcommand\restrA[2]{{
  \left.\kern-\nulldelimiterspace 
  #1 
  \vphantom{\big|} 
  \right|_{#2} 
  }}

\newcommand\restrB[2]{\ensuremath{\left.#1\right|_{#2}}}

\def\restr#1#2{\restrB{#1}{#2}}

\def\te{\theta}

\def\D{\Delta}
\def\term#1{{\it #1}}

\def\et(#1){ (#1)}

\newtheorem{statement}{Statement}[section]
\newtheorem{claim}{Fact}
\newtheorem{proposition}{Proposition}[section]
\newtheorem{theorem}{Theorem}[section]
\newtheorem*{theorem*}{Theorem}

\newtheorem*{lemma*}{Lemma}
\newtheorem{cor}{Corollary}[section]

\theoremstyle{definition}
\newtheorem{example}{Example}

\newtheorem{question}{Question}

\theoremstyle{remark}

\newtheorem*{note*}{Remark}

\def\gd/{$G_\delta$}

\def\rarr{\Rightarrow}
\def\larr{\Leftarrow}
\def\lrarr{\Leftrightarrow}

\def\DD{\operatornamewithlimits{%
  \mathchoice{\vcenter{\hbox{\huge \Delta}}}
             {\vcenter{\hbox{\Large \Delta}}}
             {\mathrm{\Delta}}
             {\mathrm{\Delta}}}}
             
\def\DD{\operatornamewithlimits{\mbox{\hbox{\huge \Delta}}}}

\def\DD{\operatornamewithlimits{\Delta}}

\def\DD{\operatornamewithlimits{%
  \mathchoice{\vcenter{\hbox{\huge $\Delta$}}}
             {\vcenter{\hbox{\Large $\Delta$}}}
             {\mathrm{\Delta}}
             {\mathrm{\Delta}}}}

\def\DD{\operatornamewithlimits{%
  \mathchoice{\vcenter{\hbox{\huge $\bigtriangleup$}}}
             {\vcenter{\hbox{\Large $\bigtriangleup$}}}
             {\mathrm{\bigtriangleup}}
             {\mathrm{\bigtriangleup}}}}

\def\diag{\DD}

\def\si{\sigma}

\def\Q{\mathbb Q}

\def\cf{\operatorname{cf}}
\def\rev{\operatorname{rev}}

\def\lex#1#2{\mathfrak{L}(#2^{#1})}
\def\lex#1#2{#2^{#1}}

\def\blex#1#2#3#4{\mathfrak{B}(#1,#2;#4^{#3})}
\def\blex#1#2#3#4{\mathfrak{B}_{#4^{#3}}(#1,#2)}
\def\blex#1#2#3#4{\mathfrak{B}_{#4}(#1,#2)}

\def\dom{\operatorname{Dom}}

\def\homg#1,#2#3{\mathfrak{D}_{#1}(#2,#3)}
\def\homm#1{{\mathfrak{D}(#1)}}

\def\symblecnc{\mathbin{{}^\frown}}

\def\wob/{$WOB$}

\def\lL(#1,#2){(-\infty,#1)}
\def\lR(#1,#2){(#1,+\infty)}

\def\B{2}
\def\B{\mathsf{2}}
\def\B{\mathbf{2}}

\def\cjx/#1,#2,#3/{{\rm (J${}_{#2,#3}^{#1}$)}}
\def\cj/#1,#2/{{\rm (J${}_{#1,#2}$)}}

\def\cjx/#1,#2,#3/{{\rm (${\mathrm J}_{#2,#3}^{#1}$)}}
\def\cj/#1,#2/{\cjx/1,#1,#2/}

\def\cJ{{\mathcal J}}

\def\ilcomma/{\boldsymbol{,}}
\def\ilsymlp{\boldsymbol{(}}
\def\ilsymrp{\boldsymbol{)}}
\def\ilsymlc{\boldsymbol{[}}
\def\ilsymrc{\boldsymbol{]}}

\def\ilfont#1{\boldsymbol{#1}}

\def\ilfontX#1{\mathrm{\text{\texttt{\textbf{\large #1}}}}}
\def\ilfontX#1{{\text{\texttt{\textbf{\large \upshape #1}}}}}

\def\ilfont#1{\raisebox{-1.2pt}{\ilfontX{#1}}}
\def\ilfontC#1{\raisebox{-0.7pt}{\ilfontX{#1}}}

\def\ilcomma/{\ilfontC{,}}
\def\ilsymlp{\ilfont{(}}
\def\ilsymrp{\ilfont{)}}
\def\ilsymlc{\ilfont{[}\,}
\def\ilsymrc{\,\ilfont{]}}

\def\ilpp(#1,#2){\ilsymlp#1\ilcomma/ #2\ilsymrp}
\def\ilcp[#1,#2){\ilsymlc#1\ilcomma/ #2\ilsymrp}
\def\ilpc(#1,#2]{\ilsymlp#1\ilcomma/ #2\ilsymrc}
\def\ilcc[#1,#2]{\ilsymlc#1\ilcomma/ #2\ilsymrc}

\def\lL(#1,#2){\ilpp(-\infty, #1)}
\def\lR(#1,#2){\ilpp(#1, +\infty)}

\def\jumps#1{\mathfrak{J}(#1)}

\def\nom{{n\in\om}}
\def\sq#1#2{(#1)_{#2}}

\def\lexab{\lex\la\B}

\def\sS{\mathbf{S}}
\def\Rs{\mathbf{R}}
\def\Hs{\mathbf{H}}

\def\wS{\wt\sS}
\def\wcS{\wt\cS}

\def\oip#1{\cH_{+}(#1)}
\def\oim#1{\cH_{-}(#1)}
\def\oipm#1{\cH_{\pm}(#1)}

\def\mns{0_{\om_1}}
\def\mxs{1_{\om_1}}

\def\er{\cB_r}
\def\el{\cB_l}
\def\ea{\cB_*}

\def\ex{\cE}

\begin{document}

\begin{frontmatter}

\title{Homogeneous linearly ordered spaces}

\author[inst1,inst2]{Anton Lipin}
\ead{tony.lipin@yandex.ru}
\address[inst1]{N.N.\ Krasovskii Institute of Mathematics and Mechanics, Ural Branch of the Russian Academy of Sciences, Yekaterinburg, Russia}
\address[inst2]{Ural Federal University, ul. Mira 19, Yekaterinburg, 620002, Russia}

\author[inst3]{Evgenii\corref{cor} Reznichenko}
\ead{erezn@inbox.ru}
\address[inst3]{Department of General Topology and Geometry, Mechanics and  Mathematics Faculty, M.~V.~Lomonosov Moscow State University, Leninskie Gory 1, Moscow, 199991 Russia}

\cortext[cor]{Corresponding author}



\begin{abstract}
Every compact subspace $K$ of a homogeneous generalized ordered (GO) space $X$ satisfies $\chi(K)\leq\om_1$ and $|K|\leq 2^{\om_1}$. If $\chi(K)>\om$ for some such $K$, then $\chi(X)=\om_1$ and $\pi\chi(X)=\om$.
We construct a homogeneous $\si$-compact linearly ordered topological space (LOTS) $\Hs$ with a compact subspace $K$ such that $\chi(x,K)=\om_1$ for every $x\in K$ and $c(K)=w(K)=|K|=2^{\om_1}$, where $c(K)$ denotes cellularity. Thus both bounds are attained by the same compact space.
Every semitopological group that is a GO space is hereditarily paracompact and has metrizable compact subspaces. If it is not a $P$-space, it is also first countable and submetrizable.
A semitopological group that is a LOTS, or a topological group that is a GO space, is either metrizable or is a $P$-space.
We also show that in an order-homogeneous LOTS every compact subset is first countable.
\end{abstract}
\end{frontmatter}

\section{Introduction} \label{sec:intro}

A space $X$ is called \term{homogeneous} if for all $x,y\in X$ there exists a self-homeomorphism $f: X\to X$ such that $f(x)=y$.
The theory of homogeneous spaces is a large and active area of general topology \cite{Arhangelskii_vanMill}. Compact homogeneous spaces are of particular interest, as are compact subsets of homogeneous spaces. Compact subsets of extremally disconnected topological groups are finite \cite{arh67}; finiteness also holds in homogeneous extremally disconnected spaces \cite{Arhangelskii2018,rezn2020}. Let $Y$ be extremally disconnected and let $Z$ be a homogeneous subspace of $Y^3$. Then every compact subspace of $Z$ is finite \cite{rezn2020}. Under the continuum hypothesis ($\mathrm{CH}$), the same conclusion holds with $Y^n$ in place of $Y^3$ for any positive integer $n$; for $Y^\om$ the corresponding compact subspaces are metrizable \cite{rezn2020}. A separate ZFC result in \cite{rezn2020} gives finiteness for compact \emph{homogeneous} subspaces of $Y^n$ themselves. Under Martin's axiom, the finite- and countable-power conclusions also hold for $\beta\om$-spaces \cite[Corollaries 1 and 2]{Groznova2023ru}.

In this paper we study linearly ordered topological spaces (LOTS), their subspaces (GO spaces), and topological properties of homogeneous LOTS and GO spaces.
A homogeneous GO space of pointwise countable type has countable character \cite[Corollary 29]{Arhangelskii2005}.
In a GO space, a point $x$ has countable character whenever it lies in a compact set $K$ with a countable outer base in the ambient space and $\chi(x,K)\leq\om$.

The main technical tool of the paper is a characterization of points of a GO space via linearly ordered neighbourhood bases. We call a point $x$ of a space $X$ a \wob/-point if it has a neighbourhood base that is linearly ordered by inclusion; for GO spaces this condition turns out to be equivalent to the character and the $\pi$-character coinciding at that point (Proposition \ref{p:1+2}). Together with the spectrum $\sp_t(x,X)$ of the left and right tightnesses at a point, this characterization reduces the computation of cardinal invariants of homogeneous GO spaces to the analysis of the tightnesses $\cf_l$, $\cf_r$ (Section \ref{sec:cfgos}), which is what allows us to obtain the main results of the paper.

If $K$ is a compact subspace of a homogeneous GO space $X$, then $\chi(K)\leq\om_1$ and $|K|\leq 2^{\om_1}$. Here $\chi(K)$ denotes the character of $K$ in its subspace topology. If $\chi(K)>\om$ for some such $K$, then $\chi(X)=\om_1$ and $\pi\chi(X)=\om$ (Theorem \ref{t:main:1}).
For every uncountable regular cardinal $\la$, there is a homogeneous countably compact LOTS $X_\la$ with $\chi(X_\la)=\la$ and $\pi\chi(X_\la)=\om$ (Example \ref{e:3}). For $\la=\om_1$, the ordinal space $\om_1+1$ embeds in $X_{\om_1}$ (Example \ref{e:4}).

Both bounds are attained in a single compact subspace of a homogeneous $\si$-compact LOTS $\Hs$. We first construct a compact LOTS $\sS$ with $c(\sS)=2^\om$ (Proposition \ref{p:bsc:1-1}). Its first extension is a compact subspace $K$ of $\Hs$ satisfying $\chi(x,K)=\om_1$ for every $x\in K$ and $c(K)=w(K)=|K|=2^{\om_1}$ (Example \ref{e:5}).

A separate line of results concerns the stronger condition of \emph{order-homogeneity} of a space (when homogeneity is realized by self-homeomorphisms that preserve the order). It is proved that in an order-homogeneous LOTS every compact subset is first countable (Theorem \ref{t:cssoh:1}); this is substantially stronger than the bound $\chi(K)\leq \om_1$ valid for merely homogeneous spaces.

Let $G$ be a semitopological group that is a GO space. Then $G$ is hereditarily paracompact (Theorem \ref{t:main:1+1}, which strengthens \cite[Corollary 2.7.]{BuzyakovaVural2014}), and compact subsets of $G$ are metrizable (Corollary \ref{c:main:1}).
If $G$ is not a $P$-space, then it is first countable and submetrizable (Theorem \ref{t:main:2}). Under this hypothesis $G$ is metrizable if it is a topological group, or if it is a LOTS, that is, its topology is generated by a linear order (Theorems \ref{t:main:2+1} and \ref{t:main:3}).

This result cannot be generalized to monotonically normal groups --- there are many examples of separable, non-first-countable, non-metrizable monotonically normal topological groups \cite{GartsideReznichenko2000nmfs}. Monotonically normal paratopological groups are hereditarily paracompact \cite[Theorem 2.6]{BuzyakovaVural2014}, and hence compact subsets of such groups have countable tightness \cite[Theorem 3.5]{WilliamsZhou1998} (Theorem \ref{t:main:3+1}). Compact subsets of monotonically normal topological groups are metrizable \cite[Theorem 2a]{Gartside1998}.
LOTS topological groups can have arbitrarily large character (Proposition \ref{p:h:tg1}).


Paper \cite{BennettBurkeLutzer2013} considers GO and LOTS spaces equipped with algebraic structures weaker than a group structure.

Section \ref{sec:qe} asks whether $X_{\om_1}$ and $\Hs$ admit right-topological group structures, as the double arrow does \cite{Banakh2008}. It also asks whether $\om_1$ bounds the tightness of compact subspaces of power-homogeneous GO spaces or of (power-)homogeneous monotonically normal spaces. We discuss compact subsets of monotonically normal groups and stratifiability questions for quasitopological groups related to questions of Collins \cite[Problem 6]{Collins1996} and Shkarin \cite[Remarks (2)]{Shkarin2004}.

The paper is organized as follows. Section \ref{sec:defs} introduces the basic definitions and notation. Section \ref{sec:cfgos} studies cardinal invariants of GO spaces and introduces the \wob/-point technique. Section \ref{sec:hcomp} proves the main Theorem \ref{t:main:1} on compact subsets of homogeneous GO spaces, together with the results on semitopological GO groups. Sections \ref{sec:lex}--\ref{sec:blexhome} are devoted to constructions based on the lexicographic order on products of linearly ordered sets and on the space of binary sequences, which are later used to build examples. Section \ref{sec:exos} constructs homogeneous extensions of LOTS, and Section \ref{sec:bcs} uses them to build $\Hs$ from the compact LOTS $\sS$. Section \ref{sec:qe} contains examples illustrating the sharpness of the main bounds, together with open questions.


\section{Terminology and definitions} \label{sec:defs}

Throughout the paper, all topological spaces are assumed to be Hausdorff.

\subsection{Linearly ordered sets and spaces}

The order topology $\tp_{<}$ on a linearly ordered set $(X,<)$ has a base consisting of $X=\ilpp(-\infty,+\infty)$, the \term{open intervals} $\ilpp(a,b)=\set{x\in X: a<x<b}$, and the open rays $\ilpp(-\infty,b)=\set{x\in X: x<b}$ and $\ilpp(a,+\infty)=\set{x\in X: a<x}$, for $a,b\in X$. We say that $\tp_{<}$ is generated by the linear order $<$.
We shall also use notation analogous to that for real numbers, e.g. $\ilcp[a,b)=\set{x\in X: a\leq x <b}$, and similarly for other interval types.

A subset $M\subset X$ is called \term{convex} if $a,b\in M$ and $a<b$ imply $\ilcc[a,b]\subset M$.

A topological space $X$ is called a \term{linearly ordered topological space} (LOTS) if there is a linear order on $X$ that generates the topology of $X$.
Subspaces of a LOTS are called \term{generalized ordered} (GO) topological spaces.

In this paper we shall also use LOTS to mean a linearly ordered set $(X,<)$ equipped with the topology $\tp_<$.

A LOTS $X$ is called \term{bounded} if $X$ has a minimum and a maximum.


A map $f: X\to Y$ between linearly ordered sets is called \term{increasing} (\term{decreasing}) if $f(a)\leq f(b)$ ($f(a)\geq f(b)$) for $a,b\in X$ with $a<b$. A map $f$ is called \term{strictly increasing} (\term{strictly decreasing}) if $f(a)< f(b)$ ($f(a)> f(b)$) for $a,b\in X$ with $a<b$. A map $f$ is called \term{(strictly) monotone} if $f$ is a (strictly) increasing or (strictly) decreasing map. A map $f$ is called an \term{order isomorphism} if $f$ is an increasing bijection. A map $f$ is called an \term{order reversing isomorphism} if $f$ is a decreasing bijection.

Note that an order isomorphism and an order reversing isomorphism are homeomorphisms between the LOTS $X$ and $Y$.

For a linearly ordered set $(X,<)$ denote by $\rev X=(X,<_{\rev})$ the set $X$ with the reverse order, $x<_{\rev} y$ if $y<x$. An order reversing isomorphism $f: X\to Y$ is an order isomorphism of the LOTS $\rev X\to Y$ and $X\to \rev Y$.

By an \term{order automorphism} we mean an order isomorphism from $X$ to itself.
By an \term{order reversing automorphism} we mean an order reversing isomorphism from $X$ to itself.
Denote
\begin{align*}
\oip X &= \set{f: f\text{ is an order automorphism of }X},
\\
\oim X &= \set{f: f\text{ is an order reversing automorphism of }X},
\\
\oipm X &= \oip X \cup \oim X.
\end{align*}
If $|X|\geq 2$ and $\oim X\neq\es$, then $\oip X$ is a subgroup of index $2$ in $\oipm X$. For $|X|\leq 1$ these groups coincide.
We call a LOTS $X$ \term{order-homogeneous} if $\oip X$ acts transitively on $X$, that is, for any $x,y\in X$ there is an order isomorphism $f: X\to X$ with $f(x)=y$.
We call a LOTS $X$ \term{$\pm$-order-homogeneous} if $\oipm X$ acts transitively on $X$, that is, for any $x,y\in X$ there is either an order isomorphism $f: X\to X$ or an order reversing isomorphism $f: X\to X$ with $f(x)=y$. Clearly,
\[
\text{order-homogeneity }\rarr \text{ $\pm$-order-homogeneity }\rarr \text{homogeneity.}
\]

Let $X$ and $Y$ be LOTS. On the product $X\times Y$ we consider the \term{lexicographic} order: $(a,b)<(c,d)$ if $a<c$, or $a=c$ and $b<d$. By the \term{sum} $X+Y$ of the LOTS $X$ and $Y$ we mean the disjoint union $X \sqcup Y$ with the order such that $X$ and $Y$ are order-embedded in $X+Y$ and $x<y$ whenever $x\in X$ and $y\in Y$.


A set $M\subset X$ is called \term{cofinal} if for every $x\in X$ there is $y\in M$ with $x\leq y$.
The \term{cofinality} $\cf X$ of a linearly ordered set $X$ is the least cardinality of a cofinal subset of $X$. If $X$ has a maximal element, then $\cf X=1$. If $X$ is the empty set, we set $\cf X=0$.
Throughout, the term \term{regular cardinal} means an infinite regular cardinal. Thus $\cf X$ is either $0$, $1$, or a regular cardinal. In each case there is an increasing cofinal sequence $\set{x_\al:\al < \cf X}\subset X$.
For $x\in X$ denote by $\cf_l(x,X)$ the cofinality of the set
$\lL(x,X)$  
and by $\cf_r(x,X)$ the cofinality of the set
$\lR(x,X)$  
in the reverse order, $\cf_r(x,X) = \cf_l(x,\rev X)=\cf\rev \lR(x,X)$.

A pair $(A,B)$ of subsets of $X$ is called a \term{cut} of the linearly ordered set $X$ if $A\cup B=X$, $A\cap B=\es$, and $a<b$ for all $a\in A$ and $b\in B$. A cut $(A,B)$ is called \term{proper} if $A\neq\es$ and $B\neq\es$. A proper cut is called a \term{gap} if $A$ has no maximum and $B$ has no minimum. A proper cut is called a \term{jump} if $A$ has a maximum and $B$ has a minimum. A jump is uniquely determined by the pair of points $a=\max A<b=\min B$, for which the interval $\ilpp(a,b)$ is empty. We call the point $a$ the \term{left jump point} and the point $b$ the \term{right jump point}. A left or right jump point will be called a \term{jump point}. A point $a$ is a left jump point if and only if $\cf_r(a,X)=1$, and a point $b$ is a right jump point if and only if $\cf_l(b,X)=1$. A whole jump is recovered from its left and right points: 
$A=\ilpc(-\infty,a]$ and $B=\ilcp[b,+\infty)$.
Denote 
\[
\jumps X = \set{(a,b)\in X\times X:  a<b,\   \ilpp(a,b)=\es}.
\]

\subsection{Topological spaces}

A space $X$ is called \term{monotonically normal} if for every point $x\in X$ and every neighbourhood $U$ of $x$ a neighbourhood $g(x,U)$ of $x$ is defined, in such a way that $g(x,U)\cap g(y,V)=\es$ whenever $x\notin V$ and $y\notin U$. Every GO space is monotonically normal \cite[Theorem 5.21]{gru1984}. We call a space $X$ \term{$p$-homogeneous} (power homogeneous) if $X^\tau$ is homogeneous for some nonzero cardinal $\tau$. A family of open sets $\cU$ of a space $X$ is called an \term{outer base} of a set $F\subset X$ if $F\subset U$ for every $U\in\cU$, and for every neighbourhood $W\supset F$ of $F$ there is $U\in\cU$ with $U\subset W$.
A space $X$ has \term{pointwise countable type} if for every $x\in X$ there is a compact set $K\subset X$ with $x \in K$ that has a countable outer base. An outer base of a singleton $\sset x$ is called a base of the point $x$.
For a cardinal $\tau$, a set $G\subset X$ is of \term{type $G_\tau$} if $G=\bigcap \cU$ for some family of open sets $\cU$ of cardinality at most $\tau$. Sets of type $G_\om$ are called sets of \term{type $G_\de$}.
Denote by $\D_X=\set{(x,x):x\in X}$ the diagonal of $X\times X$.

Let $x\in X$ be a non-isolated point of a space $X$.
The \term{character} $\chi(x,X)$ of the point $x$ is the least cardinality of a base of $x$.
A family $\cU$ of nonempty open subsets of $X$ is called a \term{$\pi$-base} of a point $x\in X$ if for every neighbourhood $V$ of $x$ there is $U\in\cU$ with $U\subset V$.
The \term{$\pi$-character} $\pi\chi(x,X)$ of the point $x$ is the least cardinality of a $\pi$-base of $x$.
The \term{pseudocharacter} $\psi(x,X)$ of the point $x$ is the least cardinality of a family $\cU$ of neighbourhoods of $x$ with $\sset x=\bigcap\cU$.
Set $P(x,X)$ \cite{JuhaszNyikosSzentmiklossy2005} to be the least cardinality of a family $\cU$ of neighbourhoods of $x$ with $x\notin\Int \bigcap\cU$.
Set $T(x,A)=\min\set{|M|: M\subset A,\ x\in\cl M}$ for $x\in \cl A\subset X$,
$\sp_t(x,X)=\set{T(x,A): A\subset X\text{ and }x\in \cl A \setminus A}$,
$t(x,X)=\sup (\sp_t(x,X)\cup\sset \om)$,  
$\sp_t(X)=\bigcup\set{\sp_t(x,X): x\in X}$. 
Define the \term{weak tightness} $wt(x,X)=T(x,X\setminus \sset x)$ \cite{WilliamsZhou1998}.

If $x$ is an isolated point, set $\chi(x,X)=\pi\chi(x,X)=\psi(x,X)=t(x,X)=P(x,X)=wt(x,X)=1$ and $\sp_t(x,X)=\es$.

We define the following cardinal invariants of a space $X$:
the \term{character}
$\chi(X)=\sup \set{\chi(x,X): x\in X}$,
the \term{$\pi$-character}
$\pi\chi(X)=\sup \set{\pi\chi(x,X): x\in X}$,
the \term{pseudocharacter}
$\psi(X)=\sup \set{\psi(x,X): x\in X}$,
$P(X)=\sup\set{P(x,X): x\in X}$, 
the \term{tightness}
$t(X)=\sup\set{t(x,X): x\in X}$, 
the \term{weak tightness}
$wt(X)=\sup\set{wt(x,X): x\in X}$,
the \term{diagonal number}
$\D(X)=\min\{\,\tau\geq 1:$ the diagonal $\D_X$ in $X\times X$ has type $G_\tau\,\}$
of the space $X$.
We write $c(X)$ for the \term{cellularity} of $X$, the supremum of the cardinalities of families of pairwise disjoint nonempty open subsets of $X$.

A sequence $\set{x_\al:\al<\tau}\subset X$ is called a \term{free sequence} if $\cl{\set{x_\al:\al<\la}}\cap \cl{\set{x_\al:\la \leq \al<\tau}}=\es$ for $\la<\tau$.

We call a point $x$ a \term{\wob/-point} if $x$ has a neighbourhood base $\cU$ that is linearly ordered by inclusion, that is, if $A,B\in\cU$, then either $A\subset B$ or $B\subset A$.
A \wob/-point has a well-ordered base: from a linearly ordered base one may take a cofinal well-ordered subfamily $\set{U_\al:\al<\la}\subset \cU$, where $\la=\cf \cU$ and $U_\be\subset U_\al$ for $\al<\be<\la$.

Let $\tau$ be a cardinal. A point $x$ of a space $X$ is called a \term{$P_\tau$-point} if the intersection $\bigcap\cU$ of any family $\cU$, $|\cU|<\tau$, of neighbourhoods of $x$ is a neighbourhood of $x$. A space $X$ is called a \term{$P_\tau$-space} if every point of $X$ is a $P_\tau$-point. Every space is a $P_\om$-space. A $P_{\om_1}$-point is called a \term{$P$-point}. A $P_{\om_1}$-space is called a \term{$P$-space}.

Let $G$ be a group with a topology. The group $G$ is called \term{right-topological} if all right translations $\rho_g: G\to G$, $h\mapsto hg$ are continuous. The group $G$ is called \term{semitopological} if multiplication in $G$ is separately continuous. The group $G$ is called \term{quasitopological} if multiplication is separately continuous and taking the inverse in $G$ is continuous. The group $G$ is called \term{paratopological} if multiplication in $G$ is continuous. The group $G$ is called \term{topological} if multiplication and taking the inverse in $G$ are continuous.


\section{Cardinal invariants of GO spaces} \label{sec:cfgos}

\begin{proposition}\label{p:1}
Let $X$ be a Hausdorff space and $x\in X$.
\begin{enumerate}
\item
If $X$ is a homogeneous space, then $\sp_t(x,X)=\sp_t(X)$.
\item
If $x\in K\subset X$, then $\sp_t(x,K)\subset \sp_t(x,X)$ and $\sp_t(K)\subset \sp_t(X)$.
\item
If $X$ is an infinite compact space, then $\sp_t(X)$ contains all infinite regular cardinals $\mu<t(X)$.
\item
If $X$ is an infinite compact space, then $\om\in \sp_t(X)$.
\item
If $x$ is a non-isolated point, then
\begin{align*}
wt(x,X)&=\min \sp_t(x,X),
&
t(x,X)&=\sup \sp_t(x,X).
\end{align*}
\item
$|\sp_t(x,X)|\leq 1$ if and only if $wt(x,X)=t(x,X)$.
\end{enumerate}
\end{proposition}
\begin{proof}
(1) and (2) are obvious.

(3) Let $\mu<t(X)$ be an infinite regular cardinal. Since $X$ is compact, $t(X)$ is the supremum of the lengths of free sequences in $X$ \cite[Theorem 7.12]{Hodel1984handbook}. Hence there is a free sequence of length greater than $\mu$, and its initial segment $\set{x_\al:\al<\mu}$ is also a free sequence; that is, $\cl{\set{x_\be:\be<\al}}\cap\cl{\set{x_\be:\al\leq\be<\mu}}=\es$ for $\al<\mu$. Let $y$ be a complete accumulation point of $A=\set{x_\al:\al<\mu}$ (it exists since $X$ is compact), that is, every neighborhood of $y$ contains $\mu$ points of $A$. Then $y\notin A$: if $y=x_\ga$, then by freeness the point $y$ has a neighborhood containing at most $|\ga|+1<\mu$ points of $A$. If $y\in\cl B$ for some $B\subset A$ with $|B|<\mu$, then, since $\mu$ is regular, $B\subset\set{x_\be:\be<\al}$ for some $\al<\mu$; but $y\in\cl{\set{x_\be:\al\leq\be<\mu}}$, which contradicts freeness. Hence $\mu=T(y,A)\in \sp_t(X)$.

The supremum in Hodel's theorem is not always attained, therefore the value $\mu=t(X)$ has to be excluded in (3). For example, let $\la$ be an uncountable regular limit cardinal and let $K$ be the one-point compactification of the topological sum of the spaces $\mu+1$, where $\mu<\la$ is an infinite regular cardinal. Then $t(K)=\la$, but $\la\notin\sp_t(K)$: the tightness of every point of $K$ except the added point $p$ is less than $\la$, and $T(p,A)\leq\om$ whenever $p\in\cl A\setminus A$.

(4) There is a countable $A\subset X$ and $y\in \cl A\setminus A$. Then $\om=T(y,A)\in \sp_t(X)$.

(5) is obvious. (6) follows from (5).
\end{proof}

The following diagram shows the relations between the pointwise cardinal invariants. An arrow $A\to B$ means that $A(x,X)\geq B(x,X)$ for every Hausdorff space $X$ and $x\in X$.
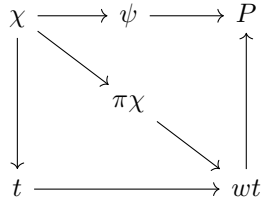
\[
\begin{tikzcd}
\chi \ar[r] \ar[dd] \ar[rd]	&	\psi \ar[r]	&	P 
\\
 & \pi\chi \ar[rd]
\\
t \ar[rr]	&&	wt \ar[uu]
\end{tikzcd}
\]


\begin{proposition}\label{p:1-1}
Let $X$ be a Hausdorff space, $x\in X$.
\begin{enumerate}
\item
If $x$ is a \wob/-point, then
\[
\chi(x,X)=\pi\chi(x,X)=\psi(x,X)=t(x,X)=wt(x,X)=P(x,X).
\]
\item
If $\chi(x,X)=P(x,X)$, then $x$ is a \wob/-point.
\end{enumerate}
\end{proposition}
\begin{proof}
(1) If $x$ is isolated, all the invariants are equal to $1$ by definition. Let $x$ be non-isolated and $\la=P(x,X)$. It suffices to show that $\chi(x,X)\leq \la$.
Let $\cU$ be a linearly ordered base of $x$ and let $\set{V_\al:\al<\la}$ be a system of neighbourhoods of $x$ with $x\notin \Int \bigcap_{\al<\la}V_\al$. Take $\cO=\set{U_\al:\al<\la}\subset \cU$ such that $U_\al\subset V_\al$ and $U_\al\supset U_\be$ for $\al<\be<\la$. Then $\cO$ is a base at $x$.

(2) If $x$ is isolated, then $\sset{\sset x}$ is a linearly ordered base of $x$. Let $x$ be non-isolated and $\la=P(x,X)$. Let $\set{W_\al:\al<\la}$ be a base of $x$ and let $\set{V_\al:\al<\la}$ be a system of neighbourhoods of $x$ with $x\notin \Int \bigcap_{\al<\la}V_\al$. 
Let $\cU=\set{U_\al:\al<\la}$ be a family of open subsets of $X$ such that $U_\al\subset W_\al\cap V_\al$ and $U_\al\supset U_\be$ for $\al<\be<\la$. Then $\cU$ is a linear base at $x$.
\end{proof}

\begin{proposition}\label{p:1+0}
Let $Y$ be a LOTS and $x\in X\subset Y$.
On $X$ we consider the order inherited from $Y$ and the subspace topology of the space $Y$.
Denote $L_*=X\cap \lL(x,Y)$, $R_*=X\cap \lR(x,Y)$, $L=L_*\cup \sset x$, $R=R_*\cup \sset x$.
\begin{enumerate}
\item
The point $x$ in $L$ is a \wob/-point. If $x\in \cl{L_*}$, then $\chi(x,L)=\cf_l(x,X)=\cf_l(x,Y)$.
\item
The point $x$ in $R$ is a \wob/-point. If $x\in \cl{R_*}$, then $\chi(x,R)=\cf_r(x,X)=\cf_r(x,Y)$.
\item
The point $x$ in $X$ is a \wob/-point if and only if one of the following conditions holds:
\begin{enumerate}
\item
$x\notin \cl{L_*}$;
\item
$x\notin \cl{R_*}$;
\item
$x\in \cl{L_*}\cap \cl{R_*}$ and $\cf_l(x,X)=\cf_r(x,X)$.
\end{enumerate}
\item
\begin{enumerate}
\item
$\sp_t(x,X)=\es$ if 
$x\notin \cl{L_*}$ and $x\notin \cl{R_*}$, that is, $x$ is an isolated point of $X$;
\item
$\sp_t(x,X)=\sset{\cf_l(x,X)}$ 
if $x\in \cl{L_*}$ and $x\notin \cl{R_*}$;
\item
$\sp_t(x,X)=\sset{\cf_r(x,X)}$ 
if $x\notin \cl{L_*}$ and $x\in \cl{R_*}$;
\item
$\sp_t(x,X)=\sset{\cf_l(x,X),\cf_r(x,X)}$ 
if $x\in \cl{L_*}$ and $x\in \cl{R_*}$.
\end{enumerate}
\end{enumerate}
\end{proposition}
\begin{proof}
Let us prove (1).
If $x\notin\cl{L_*}$, then $x$ is an isolated point of $L$, and $\sset{\sset x}$ is a linearly ordered base of $x$ in $L$. Let $x\in\cl{L_*}$, $\la=\cf_l(x,X)$, and let $\set{x_\al:\al<\la}$ be a set cofinal in $L_*$. Then
the family $\set{\ilpc(x_\al,x]:\al<\la}$ forms a base of $x$ in $\ilpc(-\infty, x]$, and
the family $\set{\ilpc(x_\al,x]\cap X:\al<\la}$ forms a base of $x$ in $L$.

Item (2) follows from (1).

Let us prove (3). ($\larr$) 
Item (a) follows from (2). Item (b) follows from (1).
Let us prove (c). Let $\la=\cf_l(x,X)$, let $\set{x_\al:\al<\la}\subset L_*$ be a sequence increasing to $x$, and let $\set{y_\al:\al<\la}\subset R_*$ be a sequence decreasing to $x$.
Then $\set{\ilpp(x_\al,y_\al)\cap X:\al<\la}$ is a linearly ordered base at the point $x$.

($\rarr$) If $x\notin \cl{L_*}$ or $x\notin \cl{R_*}$, then either (a) or (b) holds.
Consider the case $x\in \cl{L_*}\cap \cl{R_*}$. There is a linearly ordered base $\set{\ilpp(x_\al,y_\al)\cap X:\al<\la}$ of $x$, where $\la$ is some regular cardinal and $x_\al\in L_*$, $y_\al\in R_*$ for $\al<\la$. Then $\la=\cf_l(x,X)=\cf_r(x,X)$, and (c) holds.

Item (4) is obvious.
\end{proof}

The following statement follows from Proposition \ref{p:1+0}(3).

\begin{proposition}\label{p:1+0.5}
Let $X$ be a LOTS and $x\in X$. If $x$ is a minimum, a maximum, or a jump point, then $x$ is a \wob/-point.
\end{proposition}

\begin{proposition}\label{p:1+1}
Let $X$ be a GO space and $x\in X$.
Then  
\begin{align*}
\pi\chi(x,X) &= wt(x,X)= P(x,X),
&
\chi(x,X) &= \psi(x,X) = t(x,X).
\end{align*}
If $x$ is not an isolated point of $X$, then
\begin{align*}
\min \sp_t(x,X) &= \pi\chi(x,X),
&
\max \sp_t(x,X) &= \chi(x,X).
\end{align*}

\end{proposition}
\begin{proof}
The proposition is obvious if $x$ is an isolated point. Consider the case when $x$ is not an isolated point. Then $\sp_t(x,X)\neq\es$.

If $|\sp_t(x,X)|=1$, then $\min \sp_t(x,X)=\max \sp_t(x,X)$, and Proposition \ref{p:1+0}(3) implies that $x$ is a \wob/-point. The required equality then follows from Proposition \ref{p:1-1} and Proposition \ref{p:1+0}(1,2).

Consider the case $|\sp_t(x,X)|>1$.
Let $Y$ be a LOTS in which $X$ is embedded.
On $X$ we consider the order inherited from $Y$.
Denote $L_*=X\cap \lL(x,Y)$, $R_*=X\cap \lR(x,Y)$, $L=L_*\cup \sset x$, $R=R_*\cup \sset x$, $\la=\cf_l(x,X)$, $\rho=\cf_r(x,X)$.
By Proposition \ref{p:1+0}(4), $x\in \cl{L_*}\cap \cl{R_*}$, $|\sp_t(x,X)|=2$, and $\sp_t(x,X)=\sset{\la,\rho}$. Without loss of generality $\la<\rho$. Then $\la=\min \sp_t(x,X)$ and $\rho=\max \sp_t(x,X)$. By Proposition \ref{p:1+0}(1,2), $x$ is a \wob/-point in $L$ and in $R$.
By Proposition \ref{p:1-1}, $\pi\chi(x,L) = wt(x,L)= P(x,L)=\la$ and $\pi\chi(x,R) = wt(x,R)= P(x,R)=\rho>\la$. Hence $\pi\chi(x,X) = wt(x,X)= P(x,X)=\la$.

By Proposition \ref{p:1-1}, $\chi(x,L) = t(x,L)= \psi(x,L)=\la<\rho$ and $\chi(x,R) = t(x,R)= \psi(x,R)=\rho$. Hence $\chi(x,X)= \psi(x,X) = t(x,X) = \rho$.
\end{proof}

The following statement follows from Propositions \ref{p:1+0} and \ref{p:1+1}, and is also easily verified directly.

\begin{proposition}\label{p:1+1.5}
Let $X$ be a LOTS, $x\in X$ a non-isolated point, and $C=\sset{\cf_l(x,X),\cf_r(x,X)}\setminus\sset{0,1}$. Then $C\neq\es$ and
\begin{align*}
\pi\chi(x,X) &= \min C,
\\
\chi(x,X) &=  \max C.
\end{align*}
\end{proposition}

The following statement follows from Propositions \ref{p:1-1}, \ref{p:1+0} and \ref{p:1+1}.

\begin{proposition}\label{p:1+2}
Let $X$ be a GO space and $x\in X$.
The following conditions are equivalent:
\begin{enumerate}
\item
$x$ is a \wob/-point;
\item
$\chi(x,X)= \pi\chi(x,X)$;
\item
$|\sp_t(x,X)|\leq 1$.
\end{enumerate}
\end{proposition}

\begin{proposition}[{\cite[Theorem 2.5]{WilliamsZhou1998}, \cite[Theorem 3.12 (iii)]{Juhasz1992}}]\label{p:1+3}
In a monotonically normal compact space the set of \wob/-points is dense.
\end{proposition}

We give a simple proof of a corollary of this fact.

\begin{cor}\label{c:cfgos:1}
In a compact LOTS the set of \wob/-points is dense.
\end{cor}
\begin{proof}
Let $U$ be a nonempty open subset of the compact LOTS $X$. If $U$ contains an isolated point, that point is a \wob/-point. Otherwise every nonempty open subset of $U$ is infinite. Let $\cI$ be the family of nonempty intervals $\ilpp(u,v)\subset U$ with $u,v\in U$. This family is nonempty: any nonempty open convex subset of $U$ contains three points in increasing order.

Order $\cI$ by $A\prec B$ if $\cl B\subsetneq A$, and choose a maximal chain $\cU$ by the Hausdorff maximal principle. The chain has no smallest member under inclusion. Indeed, given $A\in\cI$, choose five points $p<u<c<v<q$ in $A$. Then $B=\ilpp(u,v)$ is nonempty and $\cl B\subseteq\ilcc[u,v]\subsetneq A$. If $A$ were the smallest member of $\cU$, adjoining $B$ would extend the chain. Thus every $A\in\cU$ contains the closure of a smaller member, and
\[
F=\bigcap_{A\in\cU} A=\bigcap_{A\in\cU}\cl A.
\]
The closures are nested nonempty closed subsets of $X$, so $F$ is nonempty and compact. Also $F$ is convex and lies in $U$. It has a minimum $a$ and a maximum $b$: in the absence of a minimum, the open rays $\ilpp(t,+\infty)$, $t\in F$, would cover $F$ without a finite subcover; the argument for the maximum is symmetric.

If $\ilpp(a,b)$ were nonempty, it would be an infinite open subset of $U$. Choose $u<c<v$ in $\ilpp(a,b)$. The interval $\ilpp(u,v)$ has closure in $\ilcc[u,v]\subsetneq F$, so it would extend $\cU$, a contradiction. Hence $\ilpp(a,b)=\es$.

If $a<b$, then $a$ and $b$ are jump points of $X$ lying in $U$, hence \wob/-points by Proposition \ref{p:1+0.5}. If $a=b$, then $F=\sset a$. For every open neighbourhood $O$ of $a$, some $A\in\cU$ has $\cl A\subseteq O$; otherwise the nested closed sets $\cl A\setminus O$ would have a common point in $F\setminus O$. Thus $\cU$ is a linearly ordered neighbourhood base at $a$.
\end{proof}


\begin{proposition}\label{p:1+0a}
Let $Y$ be a LOTS, $X\subset Y$.
On $X$ we consider the order inherited from $Y$ and the subspace topology of the space $Y$.
Let $L_x=\ilpp(-\infty,x)\cap X$ and $R_x=\ilpp(x,+\infty)\cap X$ for $x\in X$.
\begin{enumerate}
\item
The topology of $X$ coincides with the order topology on $X$ if and only if
for every $x\in X$ the following conditions hold:
\begin{enumerate}
\item
$\max L_x$ exists if $L_x\neq\es$ and $x\notin\cl{L_x}$;
\item
$\min R_x$ exists if $R_x\neq\es$ and $x\notin\cl{R_x}$.
\end{enumerate}
\item
The topology of $X$ coincides with the order topology on $X$ if
one of the following conditions holds:
\begin{enumerate}
\item
$x\in \cl{L_x}\cap \cl{R_x}$ for every $x\in X$;
\item
$X$ is dense in $Y$ and contains no jump points of $Y$;
\item
$X$ is compact.
\end{enumerate}
\end{enumerate}
\end{proposition}
\begin{proof}
Item (1) is verified directly without difficulty. Item (2) follows from (1).
\end{proof}

The following statements are well known \cite{Nagata1985}.

\begin{proposition}\label{p:1+4}
Let $X$ be a LOTS.
\begin{enumerate}
\item
The space $X$ is locally compact if $X$ has no gaps.
\item
The space $X$ is compact if and only if either $X=\es$ or $X$ has no gaps and has both a minimum and a maximum.
\item
The space $X$ is connected if and only if $X$ has no gaps and no jumps.
\end{enumerate}
\end{proposition}

\begin{proposition}\label{p:1+5}
Let $X$ be a compact LOTS.
The space $X$ is zero-dimensional if and only if for any $a,b\in X$, $a<b$, there is a jump $(A,B)$ such that $\ilpc(-\infty,a]\subset A$ and $\ilcp[b,+\infty)\subset B$.
\end{proposition}
\begin{proof}
($\rarr$) Let $a<b$. Since $X$ is a zero-dimensional compact space, there is a clopen $U\subset X$ such that $\ilpc(-\infty,a]\subset U$ and $\ilcp[b,+\infty)\subset X\setminus U$. Set $a'=\max U$. Then the set $\ilpc(-\infty,a']$ is clopen.
Set $b'=\min \ilpp(a',+\infty)$, $A=\ilpc(-\infty,a']$ and $B=\ilcp[b',+\infty)$.
Then $(A,B)$ is a jump and $\ilpc(-\infty,a]\subset A$ and $\ilcp[b,+\infty)\subset B$.

($\larr$) The condition implies that $X$ is totally disconnected. Since $X$ is compact, $X$ is zero-dimensional \cite[Theorem 6.2.9]{EngelkingBookRu}.
\end{proof}

\begin{proposition}\label{p:1+6}
Let $X$ be a zero-dimensional compact LOTS. A base at a point $x\in X$ is formed by open intervals of the form $\ilcc[a,b]$, where $a$ is the minimum of $X$ or a right jump point, and $b$ is the maximum of $X$ or a left jump point.
\end{proposition}
\begin{proof}
Let $U$ be an open interval containing $x$.

Let us find the point $a$.
If $\ilpp(-\infty,x)\cap U=\es$, set $a=x$. Otherwise, take $a'\in \ilpp(-\infty,x)\cap U$. By Proposition \ref{p:1+5} there exist $a'',a\in X$ such that $a'\leq a''<a\leq x$ and $\ilpp(a'',a)=\es$.

Let us find the point $b$.
If $\ilpp(x,+\infty)\cap U=\es$, set $b=x$. Otherwise, take $b'\in \ilpp(x,+\infty)\cap U$. By Proposition \ref{p:1+5} there exist $b'',b\in X$ such that $x\leq b<b''\leq b'$ and $\ilpp(b,b'')=\es$.
\end{proof}


\section{Compact subsets of homogeneous GO spaces} \label{sec:hcomp}

A locally compact $p$-homogeneous monotonically normal space has countable character \cite[Corollary 3.22]{Arhangelskii2004}.
A $p$-homogeneous GO space of pointwise countable type has countable character \cite[Corollary 29]{Arhangelskii2005}.

\begin{theorem}\label{t:main:1}
Let $X$ be a homogeneous GO space. Then
\begin{enumerate}
\item
$\chi(K)\leq\om_1$ and $|K|\leq 2^{\om_1}$ for every compact $K\subset X$;
\item
if there is a compact $K\subset X$ with $\chi(K)>\om$, then $\chi(X)=\om_1$ and $\pi\chi(X)=\om$.
\end{enumerate}
\end{theorem}
\begin{proof}
Let $K\subset X$ be compact.
By Proposition \ref{p:1+0}(4) and Proposition \ref{p:1}(1), $|\sp_t(X)|\leq 2$.
By Proposition \ref{p:1}(2), $\sp_t(K)\subset \sp_t(X)$, hence $|\sp_t(K)|\leq 2$.

Let us prove (1). If $K$ is finite, the statement is obvious. Consider the case when $K$ is infinite.
Being a compact GO space, $K$ is a LOTS, and by Proposition \ref{p:1+1}, $\chi(K)=\max \sp_t(K)$.
Let us show that $\chi(K)\leq\om_1$. Suppose that $\ka=\chi(K)>\om_1$. Then $\ka\in\sp_t(K)$ is a regular cardinal, and $\ka=\cf_l(x,K)$ or $\ka=\cf_r(x,K)$ for some $x\in K$; by symmetry, let $\ka=\cf_l(x,K)$. Take a strictly increasing sequence $\set{x_\al:\al<\ka}$ cofinal in $K\cap\ilpp(-\infty,x)$. For $\mu\in\sset{\om,\om_1}$ let $y_\mu=\sup\set{x_\al:\al<\mu}$ (it exists since $K$ is compact). Then $y_\mu\notin\set{x_\al:\al<\mu}$ and $\cf_l(y_\mu,K)=\mu$, so $\mu\in\sp_t(K)$ by Proposition \ref{p:1+0}(4). Thus $\sset{\om,\om_1,\ka}\subset\sp_t(K)$, which contradicts $|\sp_t(K)|\leq 2$. Hence $\chi(K)\leq\om_1$. Moreover, every element of $\sp_t(K)$ is an infinite regular cardinal not exceeding $\chi(K)$, so $\sp_t(K)\subset\sset{\om,\om_1}$.
By the Arhangel'skii theorem
\cite[Theorem 4.5]{Hodel1984handbook}, $|K|\leq 2^{\om_1}$.

Let us prove (2). Let $\chi(K)>\om$. By (1), $\chi(K)=\om_1\in \sp_t(K)$.
By Proposition \ref{p:1}(4), $\om\in \sp_t(K)$. Since $\sp_t(K)\subset\sset{\om,\om_1}$, we get $\sp_t(K)=\sset{\om,\om_1}$. Since $\sp_t(K)\subset \sp_t(X)$ and $|\sp_t(X)|\leq 2$, we get $\sp_t(X)=\sset{\om,\om_1}$. By Proposition \ref{p:1+1}, $\chi(X)=\om_1$ and $\pi\chi(X)=\om$.
\end{proof}

\begin{proposition}\label{p:main:1}
Let $\la$ be a regular uncountable cardinal and $C\subset \la$ a stationary subset. Then $\D(C)=\la$.
\end{proposition}
\begin{proof}
Since $|C|=\la$, we have $\D(C)\leq\la$. Suppose $\D(C)\neq\la$. Then there is a family $\cW$ of open neighbourhoods of the diagonal $\D_C$ in the square $C\times C$ such that $\bigcap\cW=\D_C$ and $|\cW|<\la$.

Let $c_0=\min C$ and $C_*=C\setminus \sset{c_0}$; the set $C_*$ is stationary. Fix $W\in\cW$. For every $x\in C_*$ the set $W$ is a neighborhood of the point $(x,x)$, and hence there is an ordinal $f_W(x)<x$ (not necessarily in $C$) such that $V_x\times V_x\subset W$, where $V_x=\set{y\in C: f_W(x)< y \leq x}$. (If $x$ is isolated in $C$, then $V_x=\sset x$ for a suitable $f_W(x)$, even when $C\cap x$ has no maximum.) By Fodor's lemma, there is $x_W<\la$ such that the set $S_W=f_W^{-1}(x_W)$ is stationary in $\la$; in particular, $S_W$ is unbounded in $\la$. Set $U_W=\set{x\in C: x> x_W}$. Let us show that $U_W\times U_W\subset W$. Let $y,z\in U_W$. Take $x\in S_W$ with $x\geq\max\sset{y,z}$. Then $f_W(x)=x_W<y,z\leq x$, so $y,z\in V_x$ and $(y,z)\in V_x\times V_x\subset W$.

Since $\la$ is a regular cardinal and $|\cW|<\la$, there is $a\in C$ such that $x_W<a$ for all $W\in \cW$ (recall that $C$ is unbounded in $\la$). Take $b\in C$, $a<b$. Then $(a,b)\notin \D_C$ and $(a,b)\in W$ for all $W\in\cW$, contradicting $\bigcap\cW=\D_C$.
\end{proof}

\begin{theorem}\label{t:main:1+1}
Let $G$ be a semitopological group that is a GO space.
Then $G$ is hereditarily paracompact and
\[
P(G) = wt(G)= \pi\chi(G) = \D(G) =
\chi(G) = \psi(G) = t(G).
\]
\end{theorem}
\begin{proof}
If $G$ is discrete, all the stated invariants are $1$ and every subspace is paracompact. Assume that $G$ is not discrete.
Let $\tau=P(G)$. By Proposition \ref{p:1+1} and the homogeneity of $G$, $wt(G)= \pi\chi(G) = \tau$.
Since $G$ is a semitopological group and $\pi\chi(G) = \tau$, we get $\D(G)\leq\tau$ (\cite[Theorem 1]{Reznichenko2023-1}, \cite[Theorem 2.4]{arh-rezn2005}, and \cite[Corollary 5.7.5]{at2009} for countable $\tau$). Since $\psi(G)\leq \D(G)$, Proposition \ref{p:1+1} and the homogeneity of $G$ give $\chi(G) = t(G) = \psi(G) \leq \D(G)\leq \tau$. Since $\pi\chi(G)\leq \chi(G)$, we get $\chi(G) = t(G) = \psi(G) = \D(G) = \tau$.

Let us prove that $G$ is hereditarily paracompact. Suppose otherwise. Since a GO space is monotonically normal, $G$ is a monotonically normal, non-hereditarily-paracompact space. By \cite[Theorem I]{BaloghRudin1992}, there is a regular uncountable cardinal $\la$ and a stationary subset $C\subset \la$ homeomorphic to some $F\subset G$. By Proposition \ref{p:main:1}, $\D(F)=\la$. Since $\D(F)\leq \D(G)$, we get $\la\leq \tau$. A stationary subset is not discrete, so let $x\in F$ be a non-isolated point of $F$. Then $\tau=wt(x,G)\leq T(x,F\setminus\sset x)<\la$, contradicting $\la\leq \tau$.
\end{proof}

\begin{theorem}\label{t:main:2}
Let $G$ be a semitopological group that is a GO space and not a $P$-space.
Then $G$ has countable character, is submetrizable, and is hereditarily paracompact.
\end{theorem}
\begin{proof}
Since $G$ is not a $P$-space, $P(x,G)=\om$ for some $x\in G$.
By Theorem \ref{t:main:1+1}, $\chi(G)=\D(G)=\om$ and $G$ is hereditarily paracompact.
Paracompact spaces with a $G_\de$-diagonal are submetrizable \cite[Corollary 2.9]{gru1984}.
Note also that GO spaces with a $G_\de$-diagonal are (hereditarily) paracompact \cite[Theorem 4.5]{Lutzer1971}.
\end{proof}

In $P$-spaces compact subsets are finite.
In submetrizable spaces compact subsets are metrizable.
Hence the following statement follows from Theorem \ref{t:main:2}.

\begin{cor}\label{c:main:1}
In semitopological GO groups compact subsets are metrizable.
\end{cor}

\begin{theorem}\label{t:main:3}
Let $G$ be a semitopological group that is a LOTS and not a $P$-space.
Then $G$ is metrizable.
\end{theorem}
\begin{proof}
By Theorem \ref{t:main:2}, $G$ has a $G_\de$-diagonal.
A LOTS with a $G_\de$-diagonal is metrizable \cite{Lutzer1969}, \cite[Theorem 2.3]{gru1984}.
Hence $G$ is metrizable.
\end{proof}

\begin{theorem}[{\cite[Theorem 2.6]{BuzyakovaVural2014}, \cite[Theorem 3.5]{WilliamsZhou1998}}]
\label{t:main:3+1}
Let $G$ be a monotonically normal paratopological group. Then $G$ is hereditarily paracompact and every compact subset of $G$ has countable tightness.
\end{theorem}
\begin{proof}
Monotonically normal paratopological groups are hereditarily paracompact \cite[Theorem 2.6]{BuzyakovaVural2014}, and hence compact subsets of such groups have countable tightness \cite[Theorem 3.5]{WilliamsZhou1998}.
\end{proof}


\begin{theorem}\label{t:main:2+1}
Let $G$ be a topological group that is a GO space and not a $P$-space.
Then $G$ is metrizable.
\end{theorem}
\begin{proof}
By Theorem \ref{t:main:2}, $G$ has countable character. Since $G$ has countable character, the Birkhoff--Kakutani theorem \cite[Theorem 3.3.12]{at2009} implies that the group $G$ is metrizable.
\end{proof}




\begin{theorem}\label{t:cssoh:1}
Let $X$ be an order-homogeneous LOTS.
Then every compact subset of $X$ is first countable.
\end{theorem}
\begin{proof}
Let $K\subset X$ be compact and suppose $\chi(K)>\om$. By Theorem \ref{t:main:1} and Proposition \ref{p:1+1.5}, some point $x\in K$ has one-sided cofinality $\om_1$ in $K$. Reversing the order if necessary, assume this is the left cofinality. Fix an increasing sequence $(a_\xi)_{\xi<\om_1}$ cofinal below $x$ in $K$. Set $z_0=a_0$, and at each successor step choose $z_{\xi+1}<x$ in $K$ above both $z_\xi$ and $a_\xi$. At each nonzero limit $\eta<\om_1$ let $z_\eta=\sup_{\xi<\eta}z_\xi$ in $K$. This supremum exists by compactness and is strictly below $x$ since $\cf_l(x,K)=\om_1$.

Setting $z_{\om_1}=x$ gives a continuous strictly increasing map $\xi\mapsto z_\xi$ from $\om_1+1$ into $K$, hence an embedding. Its sequences of lengths $\om$ and $\om_1$ converge from the left to $z_\om$ and $x$, respectively, also in $X$. Indeed, if some $y\in X$ satisfied $z_n<y<z_\om$ for all $n<\om$, then a convex open neighbourhood of $z_\om$ in $X$ not containing $y$ would contain no $z_n$, contradicting the convergence $z_n\to z_\om$ in $K$; the same argument applies to $x$. Consequently $\cf_l(z_\om,X)=\om$ and $\cf_l(x,X)=\om_1$. This contradicts order-homogeneity, since an order automorphism preserves left cofinality.
\end{proof}


\section{The product of linearly ordered spaces with the lexicographic order} \label{sec:lex}


Let $X$ be a linearly ordered set, $\la$ an ordinal. On the set $X^\la$ we consider the lexicographic order: $f<g$ for distinct $f,g\in X^\la$ if $f(\al)<g(\al)$, where $\al=\min \set{\be<\la: f(\be)\neq g(\be)}$. 
Denote $\blex f\al\la X=\set{g\in X^\la: g(\be)=f(\be)\text{ for }\be<\al}$ for $\al<\la=\dom f$.

On $X^\la$ we consider the order topology.

\begin{proposition}\label{p:2}
If $X$ is a linearly ordered set with no greatest and no least element, $\la$ is a limit ordinal, and $f\in\lex \la X$, then the family $\set{\blex f\al \la X: \al<\la}$ forms a base of open neighbourhoods of the point $f$ in the linearly ordered space $\lex \la X$.
\end{proposition}
\begin{proof}
Let us show that the set $\blex f\al \la X$ is open. Let $g\in \blex f\al \la X$. There exist $a,b\in \blex f\al \la X$ such that $a(\al)<g(\al)<b(\al)$. Then $g\in (a,b)\subset \blex f\al \la X$.

Let us show that the family $\set{\blex f\al \la X: \al<\la}$ forms a base of neighbourhoods of $f$. Let $a,b\in \lex \la X$ and $a< f <b$. Let $\al=\min\set{\de<\la: a(\de)\neq f(\de)}$, $\be=\min\set{\de<\la: b(\de)\neq f(\de)}$, $\te=\max \al,\be$. Then $f\in \blex f{\te+1} \la X\subset (a,b)$.
\end{proof}

\begin{proposition}\label{p:h:tg1}
Let $\la$ be an infinite regular cardinal and $G=\lex \la \R$ with the order topology.
Then $G$ is a LOTS topological group and $\chi(G)=\la$.
The space $G$ is a $P_\la$-space.
\end{proposition}
\begin{proof}
On $G$ we consider the group operation induced from $\R^\la$, addition of functions. By Proposition \ref{p:2}, the family of open subgroups $\set{\blex 0\al \la \R: \al<\la}$ forms a base of neighbourhoods of the zero function.
\end{proof}

The following statement follows from \cite[Proposition 1.3(b)]{AkinHrbacek2002}.

\begin{proposition}\label{p:h:1}
Let $\la$ be an ordinal and $X$ a compact LOTS.
Then the LOTS $\lex \la X$ is compact.
\end{proposition}


\section{Binary sequences with the lexicographic order} \label{sec:blex}

We study the space $2^\la$ of binary sequences with the lexicographic order. For an infinite regular cardinal $\la$, we describe the subspaces $\homg l,\ka\la$, $\homg r,\ka\la$ and $\homg lr,\ka\la$ in terms of the left and right cofinalities of their points (Proposition \ref{p:blex:1+1}). Theorem \ref{t:blex:t1} characterizes $2^\la$ up to order isomorphism among nonempty zero-dimensional compact LOTS. Its weight is $2^{<\la}=\sup_{\al<\la}2^{|\al|}$ (Proposition \ref{p:blex:3}). For a compact zero-dimensional LOTS $X$ of weight $\la$, Theorem \ref{t:blex:1} reduces the criterion to $\chi(x,X)=\la$ for every $x\in X$. We use these results in Section \ref{sec:blexhome} to prove homogeneity of the indicated subspaces and in Section \ref{sec:bcs} to construct $\sS$.

Henceforth we regard $\B$ as the linearly ordered space $\sset{0,1}$.
Let $\la$ be an ordinal.
Henceforth on $\lexab$ we consider the lexicographic order and the order topology.

Denote by $R_1$ the order-reversing bijection of $\B$ onto itself: $R_1(0)=1$ and $R_1(1)=0$.
Denote
\[
R_\la=(R_1)^\la: \lexab\to \lexab,\ f\mapsto R_1\circ f.
\]
The map $R_\la$ is
an order reversing isomorphism of $\lexab$ onto itself.

Denote by $0_\la$ (resp., $1_\la$) the function on $\la$ identically equal to zero (resp., one). Note that $0_\la =\min \lexab$ and $1_\la =\max \lexab$.

The compact LOTS $\lexab$ is zero-dimensional. Indeed, let $f<g$ and let $\delta$ be their first differing coordinate. Keep their common prefix below $\delta$. Extending it by $0$ at $\delta$ and by ones thereafter gives $a$; extending it by $1$ at $\delta$ and by zeros thereafter gives $b$. Then $f\leq a<b\leq g$ and $(a,b)$ is a jump. The assertion follows from Propositions \ref{p:h:1} and \ref{p:1+5}.

\begin{proposition}\label{p:blex:1}
Let $\la$ be a regular cardinal, $x\in \lexab$, $\ka_0=\cf x^{-1}(0)$ and $\ka_1=\cf x^{-1}(1)$. Then 
\begin{enumerate}
\item
$\cf_l(x,\lexab)=\ka_1$, $\cf_r(x,\lexab)=\ka_0$, and $\la\in\sset{\ka_0,\ka_1}$;
\item
$\chi(x,\lexab)=\la$, 
\[
\pi\chi(x,\lexab)=\begin{cases}
\min \sset{\ka_0, \ka_1}, & \ka_0\text{ and }\ka_1\text{  are infinite,}
\\
\la, & \ka_0\text{ or }\ka_1\text{ is finite.}
\end{cases}
\]
\item
the point $x$ is a \wob/-point if and only if one of the following conditions holds:
\begin{enumerate}
\item
$|x^{-1}(0)|=|x^{-1}(1)|=\la$;
\item
$x^{-1}(0) =\es$  or $x^{-1}(1) =\es$;
\item
the set $x^{-1}(0)$ has a maximum;
\item
the set $x^{-1}(1)$ has a maximum.
\end{enumerate}
\item
the point $x$ is a left jump point if and only if the set $x^{-1}(0)$ has a maximum; denoting $\mu = \max x^{-1}(0)$, in this case $x=\restr x\mu \symblecnc 0 \symblecnc 1_\la$ and $r=\restr x\mu \symblecnc 1 \symblecnc 0_\la$ is the right jump point;
\item
the point $x$ is a right jump point if and only if the set $x^{-1}(1)$ has a maximum; denoting $\mu = \max x^{-1}(1)$, in this case $x=\restr x\mu \symblecnc 1 \symblecnc 0_\la$ and $l=\restr x\mu \symblecnc 0 \symblecnc 1_\la$ is the left jump point.
\end{enumerate}
\end{proposition}
\begin{proof}
Let us prove (1). Since $\la$ is a regular cardinal and $\la=x^{-1}(0)\cup x^{-1}(1)$, either $\ka_0=\la$ or $\ka_1=\la$.
Denote $M=x^{-1}(1)$. Then $\ka_1=\cf M$.

If $\ka_1=0$, then $x=0_\la$ and $\cf_l(x,\lexab)=\ka_1=0$.

Consider the case $\ka_1=1$. Then $\mu=\max M$ exists.
Set $l=\restr x\mu \symblecnc 0 \symblecnc 1_\la$.
Then $l$ is a left jump point, $x$ is a right jump point, $l=\max \ilpp(-\infty,x)$, and $\cf_l(x,\lexab)=\ka_1=1$.

Consider the case $\ka_1\geq\om$.
Let $\set{\eta_\al:\al<\ka_1}\subset M$ be a strictly increasing cofinal sequence.
For $\al<\ka_1$ define $y_\al\in \lexab$:
\begin{equation*}
y_\al(\nu) = 
\begin{cases}
x(\nu),	& \nu\neq \eta_\al; 
\\
0,	& \nu= \eta_\al.
\end{cases}
\end{equation*}
Then the strictly increasing sequence $\set{y_\al:\al<\ka_1}\subset \lL(x,\lexab)$ converges to $x$.
Hence $\cf_l(x,\lexab)=\ka_1$.

Since $x^{-1}(0)=\left(R_\la(x)\right)^{-1}(1)$, we have
\[
\cf_r(x,\lexab)=\cf_l(R_\la(x),\lexab)=\cf \left(R_\la(x)\right)^{-1}(1)= \cf x^{-1}(0)=\ka_0.
\]

Let us prove (2). By (1) and Propositions \ref{p:1+0} and \ref{p:1+1}, $\chi(x,\lexab)=\la$. If $\ka_0$ or $\ka_1$ is finite, then $\sp_t(x,\lexab)=\sset \la$, and by Proposition \ref{p:1+1}, $\pi\chi(x,\lexab)=\la$. Otherwise, $\sp_t(x,\lexab)=\sset{\ka_0,\ka_1}$, and by Proposition \ref{p:1+1}, $\pi\chi(x,\lexab)=\min \sset{\ka_0,\ka_1}$.

Let us prove (3). By (2), the disjunction of conditions (a), (b), (c), (d) is equivalent to $\pi\chi(x,\lexab)=\chi(x,\lexab)$. By Proposition \ref{p:1+2}, the latter equality is equivalent to $x$ being a \wob/-point.

Let us prove (4). The point $x$ is a left jump point if and only if $\cf_r(x,\lexab)=1$, that is, by (1), $\cf x^{-1}(0)=\ka_0=1$. The condition $\cf x^{-1}(0)=1$ is equivalent to $x^{-1}(0)$ having a maximal element. Let $\mu = \max x^{-1}(0)$. Then $x=\restr x\mu \symblecnc 0 \symblecnc 1_\la$, and the right jump point $r=\min \lR(x,X)$ equals $\restr x\mu \symblecnc 1 \symblecnc 0_\la$.

Item (5) follows from (4).
\end{proof}

For a regular cardinal $\la$ and a cardinal $\ka$ set
\begin{align*}
\homg l,\ka\la &= \set{ x\in \lexab: \cf_l(x,\lexab) = \ka \text{ and } \cf_r(x,\lexab) = \la},
\\
\homg r,\ka\la &= \set{ x\in \lexab: \cf_r(x,\lexab) = \ka  \text{ and } \cf_l(x,\lexab) = \la},
\\
\homg lr,\ka\la &= \homg l,\ka\la \cup \homg r,\ka\la,
\\
\homm\la &= \set{ x\in \lexab: \cf_r(x,\lexab) = \la  \text{ and } \cf_l(x,\lexab) = \la},
\end{align*}

Clearly, $\homm\la = \homg l,\la\la = \homg r,\la\la = \homg lr,\la\la$.

Note that by Proposition \ref{p:blex:1}(1), if $\ka$ is a regular cardinal and $\ka<\la$, then $\homg l,\ka\la = \set{ x\in \lexab: \cf_l(x,\lexab) = \ka}$ and $\homg r,\ka\la = \set{ x\in \lexab: \cf_r(x,\lexab) = \ka}$.

\begin{proposition}\label{p:blex:1+1}
Let $\la$ be a regular cardinal.
\begin{enumerate}
\item
$\homg l,0\la=\sset{0_\la}$ and $\homg r,0\la=\sset{1_\la}$;
\item
$\homg l,1\la$ is the set of right jump points and $\homg r,1\la$ is the set of left jump points; these sets are dense in $\lexab$;
\item
if $\ka$ is a regular cardinal with $\om\leq\ka\leq\la$, then
\begin{enumerate}
\item
the sets $\homg l,\ka\la$ and $\homg r,\ka\la$ are dense in $\lexab$;
\item
the topologies of $\homg l,\ka\la$, $\homg r,\ka\la$ and $\homg lr,\ka\la$ coincide with the order topology with respect to the order inherited from $\lexab$; accordingly, the spaces listed are LOTS;
\item
if $X\in\sset{\homg l,\ka\la,\homg r,\ka\la,\homg lr,\ka\la}$ and $x\in X$, then $\chi(x,X)=\la$ and $\pi\chi(x,X)=\ka$;
\item
the spaces $\homg l,\ka\la$, $\homg r,\ka\la$ and $\homg lr,\ka\la$ are $P_\ka$-spaces;
\end{enumerate}
\item
if $\ka< \la$ is a regular cardinal, then 
\begin{enumerate}
\item
the sets $\homg l,\ka\la$ and $\homg r,\ka\la$ are disjoint;
\item
$\homg lr,\ka\la=\set{x\in\lexab: \pi\chi(x,\lexab)=\ka}$.
\end{enumerate}
\item
the family
\begin{gather*}
\{ \homg l,0\la, \homg r,0\la, \homg l,1\la, \homg r,1\la, \homm\la \} \cup 
\\
\set{\homg l,\ka\la, \homg r,\ka\la: \ka<\la\text{ is a regular cardinal}}
\end{gather*}
is a partition of $\lexab$;
\item
if $\ka>\om$ is a regular cardinal, then compact subsets of $\homg l,\ka\la$, $\homg r,\ka\la$ and $\homg lr,\ka\la$ are finite;
\item
if $\la>\om$, then the space $\homg lr,\om\la$ is countably compact;
\item
the transfinite line $\om_1+1$ embeds in the space $\homg lr,\om{\om_1}$.
\end{enumerate}
\end{proposition}
\begin{proof}
Items (1) and (2) follow from Proposition \ref{p:blex:1} (1) and (2).

Let us prove (3). Let us prove (3)(a). Let $U$ be a nonempty open subset of $\lexab$. By Proposition \ref{p:1+6}, there exist $a,b\in \lexab$ such that $\ilcc[a,b]\subset U$ is a nonempty open subset of $\lexab$, $a$ is the minimum of $\lexab$ or a right jump point, and $b$ is the maximum of $\lexab$ or a left jump point. By Proposition \ref{p:blex:1}(1), $\cf_r(a,\lexab)=\cf_l(b,\lexab)=\la$. Since $a$ is the minimum or a right jump point, and $b$ is the maximum or a left jump point, Proposition \ref{p:blex:1}(4),(5) shows that there is $\de<\la$ such that $a(\al)=0$ and $b(\al)=1$ for all $\al\geq\de$. Define $c\in\lexab$ by $c(\al)=a(\al)$ for $\al<\de$ and, for $\al\geq\de$, as follows: if $\ka<\la$, then $c(\de+\xi)=1$ for $\xi<\ka$ and $c(\al)=0$ for $\al\geq\de+\ka$; if $\ka=\la$, then $c(\de+\xi)=1$ for even $\xi$ and $c(\de+\xi)=0$ for odd $\xi$ (that is, the values $1$ and $0$ alternate until the end). In both cases $a\leq c\leq b$, so $c\in\ilcc[a,b]\subset U$. Moreover, $\cf c^{-1}(1)=\ka$ and $\cf c^{-1}(0)=\la$, and by Proposition \ref{p:blex:1}(1), $\cf_l(c,\lexab)=\ka$ and $\cf_r(c,\lexab)=\la$, that is, $c\in\homg l,\ka\la$. Exchanging the roles of $0$ and $1$ in the definition of $c$ we obtain a point of $\ilcc[a,b]\cap\homg r,\ka\la$. Hence $\homg l,\ka\la\cap U\neq\es$ and $\homg r,\ka\la\cap U\neq\es$.

Let us prove (3)(b). The sets $\homg l,\ka\la$ and $\homg r,\ka\la$ are dense in $\lexab$ and contain no jump points and no minimum or maximum. Then (b) follows from Proposition \ref{p:1+0a}.

Item (3)(c) follows from Proposition \ref{p:blex:1} (1) and (2).

Let us prove (3)(d). Let $X\in\sset{\homg l,\ka\la,\homg r,\ka\la,\homg lr,\ka\la}$. By (3)(b), $X$ is a LOTS, and by (3)(c) and Proposition \ref{p:1+1}, $P(x,X)=\pi\chi(x,X)=\ka$ for all $x\in X$. Hence $X$ is a $P_\ka$-space.

Let us prove (4). Item (4)(a) follows from $\ka<\la$ and Proposition \ref{p:blex:1} (1) and (2).
Item (4)(b) follows from $\ka<\la$, Proposition \ref{p:blex:1} (1), (2), and Proposition \ref{p:1+1}.

Item (5) follows from the previous items.

Let us prove (6). By (3)(d), $\homg lr,\ka\la$ is a $P_\ka$-space, and since $\ka>\om$, it is a $P$-space. Hence compact subsets of $\homg lr,\ka\la$ are finite.

Let us prove (7). Let $M\subset \homg lr,\om\la$ and $|M|=\om$. Let $F$ be the closure of $M$ in $\lexab$. Let us show that $F\subset \homg lr,\om\la$. Let $x\in F\setminus M$. Then $wt(x,\lexab)=\om$. By Proposition \ref{p:1+1}, $\pi\chi(x,\lexab)=\om$. By (4)(b), $x\in \homg lr,\om\la$. Thus $F\subset \homg lr,\om\la$.

For (8), put $B=\lex{\om_1}\B$ and choose $r\in\homg r,\om{\om_1}$. Since $\cf_l(r,B)=\om_1$, fix an increasing sequence $(a_\xi)_{\xi<\om_1}$ cofinal below $r$. Choose $x_0\in\homg l,\om{\om_1}$ below $r$. At each successor step, use the density in (3)(a) to choose $x_{\xi+1}\in\homg l,\om{\om_1}$ strictly between $\max\{x_\xi,a_\xi\}$ and $r$.
At each nonzero limit $\eta<\om_1$, put $x_\eta=\sup_{\xi<\eta}x_\xi$ in $B$. This supremum exists by compactness and is below $r$ because $\cf_l(r,B)=\om_1$. The sequence below $x_\eta$ has a countable cofinal subsequence, so $\cf_l(x_\eta,B)=\om$. Proposition \ref{p:blex:1}(1) gives $\cf_r(x_\eta,B)=\om_1$, hence $x_\eta\in\homg l,\om{\om_1}$. Finally put $x_{\om_1}=r$. The map $\xi\mapsto x_\xi$ is strictly increasing and continuous at every limit ordinal, including $\om_1$ by cofinality. It is therefore a continuous injection from the compact space $\om_1+1$ into $\homg lr,\om{\om_1}$, hence an embedding.
\end{proof}

\begin{proposition}\label{p:blex:2-1}
Let $f: X\to Y$ be a surjective increasing map between LOTS $X$ and $Y$.
If $f^{-1}(y)$ is a closed interval in $X$ for every $y\in Y$, then $f$ is continuous.
\end{proposition}
\begin{proof}
It suffices to show that $f^{-1}(\lL(y,-\infty))$ and $f^{-1}(\lR(y,+\infty))$ are open for every $y\in Y$.
Let $\ilcc[u,v]=f^{-1}(y)$. Then $f^{-1}(\lL(y,-\infty))=\lL(u,-\infty)$ and $f^{-1}(\lR(y,-\infty))=\lR(v,+\infty)$.
\end{proof}

For ordinals $\ka\leq \la$ denote by
\[
\pi_\ka^\la: \lexab \to \B^\ka,\ x\mapsto \restr x\ka.
\]
the projection.

The following statement follows from Proposition \ref{p:blex:2-1}.

\begin{proposition}\label{p:blex:2}
Let $\ka\leq \la$ be ordinals. The projection $\pi^\la_\ka: \lexab \to \B^\ka$ is continuous.
\end{proposition}

\begin{theorem}\label{t:blex:t1}
Let $X$ be a nonempty zero-dimensional compact LOTS and let $\la$ be an infinite regular cardinal.
The following conditions are equivalent:
\begin{enumerate}
\item
$X$ and $\lexab$ are order isomorphic;
\item
$X$ and $\lexab$ are homeomorphic;
\item
$\chi(x,X)=\la$ for all $x\in X$ and $X$ embeds in a product $\prod_{\al<\la}X_\al$ of Hausdorff spaces such that $\chi(X_\al)<\la$ for $\al<\la$;
\item
$\chi(x,X)=\la$ for all $x\in X$ and the set $M$ of jump points of the LOTS $X$ decomposes into a union $M=\bigcup_{\al<\la} M_\al$ of subsets such that $\chi(\cl{M_\al})<\la$ for $\al<\la$.
\end{enumerate}
\end{theorem}
\begin{proof}
(1) $\rarr$ (2). Obvious.

(2) $\rarr$ (3). By Proposition \ref{p:blex:2}, $\lexab$ embeds via the map $\diag_{\al<\la} \pi_\al^\la$ into the product $\prod_{\al<\la} \lex\al\B$. It remains to note that $\chi(\lex\al\B)\leq \max\sset{1,|\al|}<\la$ for $\al<\la$.

(3) $\rarr$ (4). Assume $X\subset \prod_{\al<\la}X_\al$. Replacing each factor by the projection of $X$ onto it, we may assume that the factors are compact Hausdorff spaces; their characters are still less than $\la$. Denote $P_\be = \prod_{\al<\be}X_\al$ for $\be\leq\la$ and let $p_\be$ be the projection from $P_\la$ onto $P_\be$ for $\be<\la$.

Let $\be<\la$. Since $\la$ is regular, $\ka'=\sup\{\chi(X_\al):\al<\be\}<\la$. Set $\ka=\max\{1,\ka'+|\be|\}$. Then $\chi(P_\be)\leq\ka<\la$, including when $\be=0$ and $P_0$ is a singleton.

For $(a,b)\in\jumps X$ set
\[
\te_{a,b}=\min \set{\al<\la: p_\al(\ilpc(-\infty,a]) \cap  p_\al(\ilcp[b,+\infty))=\es}.
\]
To see that this minimum exists, put $A=\ilpc(-\infty,a]$ and $B=\ilcp[b,+\infty)$. The sets
\[
U_\xi=\set{(u,v)\in A\times B:u(\xi)\ne v(\xi)},\qquad \xi<\la,
\]
are open in $A\times B$ since the factors are Hausdorff. They cover $A\times B$ since $A$ and $B$ are disjoint. By compactness there is a finite $F\subset\la$ such that $\{U_\xi:\xi\in F\}$ still covers $A\times B$. For any $\al<\la$ with $F\subset\al$ we have $p_\al(A)\cap p_\al(B)=\es$.
Set $M_\be=\set{a,b : (a,b)\in\jumps X,\ \te_{a,b}=\be}$ for $\be<\la$. Then $M=\bigcup_{\be<\la} M_\be$. Let us show that $\chi(\cl{M_\be})<\la$.

First we show that every fiber of the map $\restr{p_\be}{\cl{M_\be}}$ contains at most two points. Let $x<y$ be points of $X$ with $p_\be(x)=p_\be(y)$. Then $M_\be\cap\ilpp(x,y)=\es$. Indeed, suppose that $m\in M_\be\cap\ilpp(x,y)$, and let $(a,b)\in\jumps X$ be a jump with $m\in\sset{a,b}$ and $\te_{a,b}=\be$. Since $b$ is the immediate successor of $a$, in both cases $m=a$ and $m=b$ we have $x\leq a$ and $b\leq y$. Thus $x\in \ilpc(-\infty,a]$ and $y\in\ilcp[b,+\infty)$, and hence $p_\be(x)=p_\be(y)\in p_\be(\ilpc(-\infty,a])\cap p_\be(\ilcp[b,+\infty))=\es$, a contradiction. Thus the open interval $\ilpp(x,y)$ is disjoint from $M_\be$, and therefore from $\cl{M_\be}$. If a fiber contained three points $x<z<y$, then $z\in\ilpp(x,y)$ would not belong to $\cl{M_\be}$. Hence $|\cl{M_\be}\cap p_\be^{-1}(t)|\leq 2$ for every $t\in P_\be$.

Now $\cl{M_\be}$ is compact, so $\restr{p_\be}{\cl{M_\be}}$ is a closed map onto $Z=p_\be(\cl{M_\be})\subset P_\be$ with fibers of cardinality at most $2$. Let $x\in\cl{M_\be}$, let $F=\cl{M_\be}\cap p_\be^{-1}(p_\be(x))$, and let $\cB$ be a base of $p_\be(x)$ in $Z$ with $|\cB|\leq\chi(P_\be)\leq\ka$. Since $F$ is finite and the map is closed, the sets $\cl{M_\be}\cap p_\be^{-1}(B)$, $B\in\cB$, form a base of the set $F$: every open $O\supset F$ contains such a set. Separating the points of $F$ by disjoint neighborhoods, we get $\chi(x,\cl{M_\be})\leq\ka$. Hence $\chi(\cl{M_\be})\leq\ka<\la$.

(4) $\rarr$ (1). Set $\cM_\al=\set{(a,b)\in \jumps X: \sset{a,b}\cap M_\al\neq\es}$ for $\al<\la$ and define the map
\[
\Om: \jumps X \to \la,\ (a,b) \mapsto \min\set{\al<\la: (a,b)\in \cM_\al}.
\]
Set $\te_{u,v}=\min  \Om(\jumps{\ilcc[u,v]})$ and $\cJ_{u,v}=\Om^{-1}(\te_{u,v})\cap \jumps{\ilcc[u,v]}$ for $u,v\in X$, $u<v$. Fix $(a_{u,v},b_{u,v})\in \cJ_{u,v}$ for every $u,v\in X$, $u<v$.

By induction on $\al<\la$, we construct increasing continuous maps $f_\al: X\to \B^\al$ such that the following conditions hold:
\begin{enumerate}
\item[(0)]
$f_\al$ is surjective, and $f_\al^{-1}(s)$ is a nonempty closed interval $\ilcc[u,v]$ of $X$ for every $s\in\B^\al$;
\item
$f_\al=\pi_\al^\be\circ f_\be$, where $\al<\be<\la$ and $\pi_\al^\be: \B^\be\to \B^\al$ is the projection;
\item
if $\al<\la$, 
$s\in \B^\al$ and $\ilcc[u,v]=f^{-1}_\al(s)$, then $f^{-1}_{\al+1}(s\fun 0)=\ilcc[u,a_{u,v}]$ and $f^{-1}_{\al+1}(s\fun 1)=\ilcc[b_{u,v},v]$.
\end{enumerate}
Base of the induction, $\al=0$. Then $\B^\al=\sset\es$; let $f_\al: X\to \sset\es$ be the constant map.

Suppose $\al=\be+1<\la$. For $s\in\B^\be$ write $f_\be^{-1}(s)=\ilcc[u,v]$. This interval cannot be a singleton: since $f_\be$ is a closed map, a singleton fiber $\sset u$ would give $\chi(u,X)\leq\chi(s,\B^\be)<\la$. Thus $u<v$ and we may split the fiber at the chosen jump $(a_{u,v},b_{u,v})$. Define $f_\al$ on $\ilcc[u,a_{u,v}]$ to have value $s\fun 0$ and on $\ilcc[b_{u,v},v]$ to have value $s\fun 1$. Both intervals are nonempty and together they exhaust the fiber. This defines an increasing surjection with nonempty closed interval fibers, compatible with $f_\be$ and hence with all earlier maps.

If $0<\al<\la$ is a limit ordinal, set $f_\al(x)=\bigcup_{\be<\al}f_\be(x)$ for $x\in X$. The maps already constructed are compatible, so this defines an increasing map to $\B^\al$. For $s\in\B^\al$ its fiber is
\[
f_\al^{-1}(s)=\bigcap_{\be<\al} f_\be^{-1}(\restr{s}{\be}).
\]
By compactness the intersection of these nested nonempty closed intervals is again a nonempty closed interval. Thus $f_\al$ is surjective also at a limit stage. In both cases Proposition \ref{p:blex:2-1} gives continuity.

The maps $f_\al$ for $\al<\la$ have been constructed.

Define the map $f: X\to \lexab$. For $x\in X$ set $f(x)(\al)=f_{\al+1}(x)(\al)$.
Since each $f_\al$ is a monotone map, $f$ is a monotone map.

For $s\in \lexab$ and $\al<\la$ denote 
\begin{align*}
\ilcc[u_\al^s,v_\al^s]&=f^{-1}_\al(\pi_\al^\la(s)), 
&
(a_\al^s,b_\al^s)&=(a_{u_\al^s,v_\al^s},b_{u_\al^s,v_\al^s}).
\end{align*}
Then
\begin{align*}
\ilcc[u_{\be+1}^s,v_{\be+1}^s] &=
\begin{cases}
\ilcc[u_{\be}^s,a_{\be}^s], &  s(\be)=0,
\\
\ilcc[b_{\be}^s,v_{\be}^s], &  s(\be)=1,
\end{cases}
\\
\ilcc[u_{\ga}^s,v_{\ga}^s] &= \bigcap_{\al<\ga}\ilcc[u_\al^s,v_\al^s],
\end{align*}
for $\be<\la$ and limit $\ga<\la$. Hence 
\[
f^{-1}(s) = \ilcc[u^s,v^s]=\bigcap_{\al<\la}\ilcc[u_\al^s,v_\al^s]\neq\es.
\]
Since $f^{-1}(s)\neq\es$ for all $s\in\lexab$, $f$ is surjective.
Since $f$ is a monotone surjection and $f^{-1}(s)$ is a closed interval for all $s\in\lexab$, Proposition \ref{p:blex:2-1} implies that $f$ is continuous.

Let us show that $f$ is a homeomorphism; for this it suffices to check that $f$ is injective. Suppose otherwise. Then $|f^{-1}(s)|>1$, and hence $u^s<v^s$ for some $s\in\lexab$.
Denote $\te^s=\te_{u^s,v^s}$ and $\te_\al=\te_{u_\al^s,v_\al^s}$ for $\al<\la$.
Then
\[
\ilcc[u^s,v^s] \subsetneq \ilcc[u^s_\be,v^s_\be] \subsetneq \ilcc[u^s_\al,v^s_\al]
\qquad\text{ and } \qquad \te_\al\leq \te_\be\leq  \te^s < \la
\]
for $\al< \be < \la$. 
Hence either $u^s_\al<u^s$ for all $\al<\la$, or $v^s<v^s_\al$ for all $\al<\la$. Assume, without loss of generality, that $u^s_\al<u^s$ for all $\al<\la$.
Since $\la$ is a regular cardinal and the sequence $\sq{\te_\al}{\al<\la}$ is bounded and increasing, the sequence $\sq{\te_\al}{\al<\la}$ stabilizes: for some $\te<\la$ and $\al_0<\la$, $\te_\al=\te$ for all $\al$ with $\al_0\leq \al < \la$. Set
\[
C=\set{ \al<\la: \al_0<\al,\ u^s_{\al_0} < u^s_{\al} < u^s_{\al+1} }
\]
The endpoints $u^s_\al$ have supremum $u^s$ and are continuous at limit stages. Since none equals $u^s$, their strict increases at successor steps are unbounded in $\la$. Hence $C$ is cofinal in $\la$. For each $\al\in C$, the chosen jump lies in $\cM_\te$ and
\[
u^s_\al\leq a^s_\al<b^s_\al=u^s_{\al+1}<u^s.
\]
Choose $p_\al\in M_\te\cap\sset{a^s_\al,b^s_\al}$ and put $S=\set{p_\al:\al\in C}$. These points are cofinal below $u^s$, so $u^s\in\cl S$. By regularity of $\la$, every subset of $S$ of cardinality less than $\la$ is bounded above by some $u^s_\beta<u^s$ and therefore does not accumulate at $u^s$. Thus $T(u^s,S)=\la$. Since $S\subset M_\te$, this gives $t(u^s,\cl{M_\te})\geq\la$, contrary to $\chi(\cl{M_\te})<\la$.
\end{proof}

\begin{proposition}\label{p:blex:ct1}
Let $\la$ be an infinite regular cardinal and $X$ a nonempty closed subset of $\lexab$.
The following conditions are equivalent:
\begin{enumerate}
\item
$X$ and $\lexab$ are order isomorphic;
\item
$\chi(x,X)=\la$ for all $x\in X$.
\end{enumerate}
\end{proposition}
\begin{proof}
(1) $\rarr$ (2) follows from Theorem \ref{t:blex:t1}.
Let us prove (2) $\rarr$ (1). Since, by Theorem \ref{t:blex:t1}, condition (3) of Theorem \ref{t:blex:t1} holds for $\lexab$ and $X\subset \lexab$, condition (3) of Theorem \ref{t:blex:t1} holds for $X$.
Hence $X$ and $\lexab$ are order isomorphic.
\end{proof}

\begin{proposition}\label{p:blex:3-1}
Let $X$ be an infinite compact zero-dimensional LOTS. Then $w(X)=|\jumps X|$.
\end{proposition}
\begin{proof}
Since $X$ is an infinite compact zero-dimensional LOTS, $\cU=\set{\ilpp(-\infty,b), \ilpp(a,+\infty): (a,b)\in \jumps X}$ forms a clopen subbase of the zero-dimensional compact space $X$. Hence $w(X)\leq|\jumps X|$. Conversely, let $\cB$ be a base of $X$. For every $(a,b)\in\jumps X$ choose $U_{a,b}\in\cB$ with $a\in U_{a,b}\subset\ilpp(-\infty,b)$. If $(a,b)\neq(a',b')$ are jumps and $a<a'$, then $b\leq a'$, so $a'\notin U_{a,b}$ and $a'\in U_{a',b'}$; hence $U_{a,b}\neq U_{a',b'}$. Thus $|\jumps X|\leq|\cB|$, and $|\jumps X|\leq w(X)$.
\end{proof}

\begin{proposition}\label{p:blex:3}
Let $\la$ be a limit ordinal. Then $w(\lexab)=c(\lexab)=\sup\set{2^\al: \al<\la}$.
\end{proposition}
\begin{proof}
By Proposition \ref{p:blex:2}, $\lexab$ embeds via the map $\diag_{\al<\la} \pi_\al^\la$ into the product $\prod_{\al<\la} \lex\al\B$. Hence $w(\lexab)\leq \tau= \sup\set{2^\al: \al<\la}$.
For $\al<\la$, under the map $\pi^\la_\al$ the preimage of each point has nonempty interior. Hence $c(\lexab)\geq \tau$. It remains to note that $c(\lexab)\leq w(\lexab)$.
\end{proof}

\begin{theorem}\label{t:blex:1}
Let $\la$ be a regular cardinal, $X$ a compact zero-dimensional LOTS with $w(X)=\la$.
The following conditions are equivalent:
\begin{enumerate}
\item
$X$ and $\lexab$ are order isomorphic;
\item
$\chi(x,X)=\la$ for all $x\in X$.
\end{enumerate}
If either of these conditions holds, then $\la=\sup\set{2^\al: \al<\la}$.
\end{theorem}
\begin{proof}
By Proposition \ref{p:blex:3}, (1) implies $\la=\sup\set{2^\al: \al<\la}$.
(1) $\rarr$ (2) follows from Theorem \ref{t:blex:t1}.
Let us prove (2) $\rarr$ (1). Since $X$ is zero-dimensional and $w(X)=\la$, $X$ embeds in the product $\prod_{\al<\la}X_\al$, where each $X_\al=\sset{0,1}$ is a two-point discrete space (this is the Tychonoff cube with the product topology, not the lexicographic space $\lexab$). Since $\chi(X_\al)=1<\la$, condition (3) of Theorem \ref{t:blex:t1} holds for $X$; this includes the case $\la=\om$. Hence $X$ and $\lexab$ are order isomorphic.
\end{proof}

The following statement follows from Theorem \ref{t:blex:1}.

\begin{cor}\label{c:blex:1}
{\rm ($\mathrm{CH}$)}
Let $X$ be a compact zero-dimensional LOTS.
The following conditions are equivalent:
\begin{enumerate}
\item
the LOTS $X$ is order isomorphic to $\lex {\om_1} \B$;
\item
$w(X)=\om_1$ and $\chi(x,X)=\om_1$ for all $x\in X$.
\end{enumerate}
\end{cor}

\begin{proposition}\label{p:blex:5}
$\lex {\om_1} \B$ does not embed homeomorphically in a homogeneous GO space.
\end{proposition}
\begin{proof}
Suppose $\lex {\om_1} \B$ embeds homeomorphically in a homogeneous GO space $X$, with $X\subset Y$ where $Y$ is a LOTS. Let $K\subset X$ be a compact subset homeomorphic to $\lex {\om_1} \B$.
By Theorem \ref{t:blex:t1}, $K$ is order isomorphic to $\lex {\om_1} \B$. Then $\cf_l(x,K)=\cf_r(x,K)=\om_1$ for some $x\in K$. Hence $\cf_l(x,X)=\cf_r(x,X)=\om_1$. By Proposition \ref{p:1+0}, $x$ is a WOB-point in $X$ and $\chi(x,X)=\om_1$. Hence $x$ is a $P$-point in $X$. By homogeneity of $X$, $X$ is a $P$-space, contradicting the fact that compact subsets of $P$-spaces are finite.
\end{proof}

\begin{theorem}\label{t:blex:2}
Let $\la$ be a regular cardinal, $a,b\in \lexab$, $a<b$.
The following conditions are equivalent:
\begin{enumerate}
\item
the LOTS $\ilcc[a,b]$ and $\lexab$ are order isomorphic;
\item
$\cf_r(a,\lexab)=\cf_l(b,\lexab)=\la$.
\end{enumerate}
\end{theorem}
\begin{proof}
Suppose that (1) or (2) holds. Then the points $a$ and $b$ are not isolated in $\ilcc[a,b]$: under (1) because $\lexab$ has no isolated points, and under (2) because $\cf_r(a,\lexab)=\cf_l(b,\lexab)=\la$ is infinite. So Proposition \ref{p:1+1.5} gives $\chi(a,\ilcc[a,b])=\cf_r(a,\lexab)$ and $\chi(b,\ilcc[a,b])=\cf_l(b,\lexab)$. For $x\in \ilpp(a,b)$, $\chi(x,\ilcc[a,b])=\chi(x,\lexab)=\la$.
Hence condition (2) is equivalent to $\chi(x,\ilcc[a,b])=\la$ for $x\in \ilcc[a,b]$.
Then (1) $\lrarr$ (2) follows from Proposition \ref{p:blex:ct1}.
\end{proof}

\begin{theorem}\label{t:blex:3}
Let $\la$ be a regular cardinal, $a,b,c,d\in \lexab$, $a<b$ and $c<d$.
The following conditions are equivalent:
\begin{enumerate}
\item
the LOTS $\ilcc[a,b]$ and $\ilcc[c,d]$ are order isomorphic and have no isolated points;
\item
$\cf_r(a,\lexab)=\cf_r(c,\lexab)\geq \om$ and
$\cf_l(b,\lexab)=\cf_l(d,\lexab)\geq \om$.
\end{enumerate}
\end{theorem}
\begin{proof}
Let 
$\al=\cf_r(a,\lexab)=\cf_r(c,\lexab)$ and
$\be = \cf_l(b,\lexab)=\cf_l(d,\lexab)$ (the cofinalities at the left and at the right ends of the intervals, respectively).

Clearly, (1) $\rarr$ (2). Let us prove (2) $\rarr$ (1).
Take $(a',b'),(c',d')\in \jumps{\lexab}$ such that $a< a'<b'< b$ and $c< c'<d'< d$.
By Proposition \ref{p:blex:1}, $\cf_l(a',\lexab)=\cf_r(b',\lexab)=\cf_l(c',\lexab)=\cf_r(d',\lexab)=\la$. Since $\ilcc[a,b]=\ilcc[a,a']\cup \ilcc[b',b]$ and $\ilcc[c,d]=\ilcc[c,c']\cup \ilcc[d',d]$, it suffices to prove that $\ilcc[a,a']$ is order isomorphic to $\ilcc[c,c']$ and $\ilcc[b',b]$ is order isomorphic to $\ilcc[d',d]$. This fact follows from the theorem being proved in the case $\al=\la$ or $\be=\la$. Let us prove (2) $\rarr$ (1) in the case $\al=\la$. The set $\homm\la$ is dense in $\lexab$ (Proposition \ref{p:blex:1+1}(3)(a) with $\ka=\la$). Take continuous strictly increasing sequences $\sq{x_\te}{\te<\be}$ and $\sq{y_\te}{\te<\be}$ in $\lexab$, that is, $x_\te=\sup\set{x_\eta:\eta<\te}$ and $y_\te=\sup\set{y_\eta:\eta<\te}$ for limit $\te<\be$, with $x_0=a$, $y_0=c$, $x_{\te+1},y_{\te+1}\in \homm\la$, $\sq{x_\te}{\te<\be}$ cofinal in $\ilcp[a,b)$, and $\sq{y_\te}{\te<\be}$ cofinal in $\ilcp[c,d)$. Such sequences exist, since $\be=\cf_l(b,\lexab)$, so a sup of fewer than $\be$ points of $\ilcp[a,b)$ is again less than $b$, and every nonempty open interval meets $\homm\la$. For a limit $\te<\be$ we have $\cf_l(x_\te,\lexab)=\cf\te<\be\leq\la$, hence $\cf_r(x_\te,\lexab)=\la$ by Proposition \ref{p:blex:1}(1); the same holds for $y_\te$. Therefore $\cf_r(x_\te,\lexab)=\cf_l(x_{\te+1},\lexab)=\cf_r(y_\te,\lexab)=\cf_l(y_{\te+1},\lexab)=\la$ for all $\te<\be$. By Theorem \ref{t:blex:2} there is an order isomorphism $f_\te: \ilcc[x_\te,x_{\te+1}]\to \ilcc[y_\te,y_{\te+1}]$ for $\te<\be$. Define an order isomorphism $f: \ilcc[a,b]\to \ilcc[c,d]$:
\[
f(x)=\begin{cases}
y_\te, & \text{if }x=x_\te\text{ for some }\te<\be,
\\
f_\te(x), & \text{if }x\in \ilpp(x_\te,x_{\te+1})\text{ for some }\te<\be,
\\
d, & \text{if }x=b
\end{cases}
\]
for $x\in \ilcc[a,b]$.
\end{proof}


\section{Homogeneous subspaces of the space of binary sequences} \label{sec:blexhome}
\begin{proposition}\label{p:blexhome:1-1}
Let $\la$ be a regular cardinal and $x,y\in \lexab$.
The following conditions are equivalent:
\begin{enumerate}
\item
$f(x)=y$ for some $f\in\oip\lexab$;
\item
$\cf_r(x,\lexab)=\cf_r(y,\lexab)$ and
$\cf_l(x,\lexab)=\cf_l(y,\lexab)$.
\end{enumerate}
\end{proposition}
\begin{proof}
Let 
$\al=\cf_l(x,\lexab)=\cf_l(y,\lexab)$ and
$\be = \cf_r(x,\lexab)=\cf_r(y,\lexab)$.
By Proposition \ref{p:blex:1}, $\la\in\sset{\al,\be}$.
Clearly, (1) $\rarr$ (2). Let us prove (2) $\rarr$ (1).
If $0\in\sset{\al,\be}$, then $x=y$ and the statement is obvious.

Consider the case $1\in\sset{\al,\be}$. Consider the case $\al=1$. Then $x$ and $y$ are right jump points and $\be=\la$. Let $x'$ and $y'$ be the corresponding left jump points. Then $\cf_r(0_\la)=\cf_l(x')=\cf_r(x)=\cf_l(y')=\cf_r(y)=\cf_l(1_\la)=\la$. By Theorem \ref{t:blex:3} there exist order isomorphisms $f_1: \ilcc[0_\la,x']\to \ilcc[0_\la,y']$ and $f_2: \ilcc[x,1_\la]\to \ilcc[y,1_\la]$. For $u\in \lexab$, set
\[
f(u)=\begin{cases}
f_1(u), & u \leq x',
\\
f_2(u), & u \geq x.
\end{cases}
\]
Then $f\in\oip{\lexab}$ and $f(x)=y$.
Consider the case $\be=1$. Then $\cf_l(R_\la(x),\lexab)=\cf_l(R_\la(y),\lexab)=1$ and $\cf_r(R_\la(x),\lexab)=\cf_r(R_\la(y),\lexab)=\la$. Hence $f'(R_\la(x))=R_\la(y)$ for some $f'\in \oip{\lexab}$. We obtain $f(x)=y$ for $f=R_\la\circ f' \circ R_\la\in \oip\lexab$.

Consider the case when $\al$ and $\be$ are infinite cardinals.
Since $0_\la<x<1_\la$, $0_\la<y<1_\la$, $\al\geq \om$, and $\be\geq \om$, Theorem \ref{t:blex:3} gives order isomorphisms $f_1: \ilcc[0_\la,x]\to \ilcc[0_\la,y]$ and $f_2: \ilcc[x,1_\la]\to \ilcc[y,1_\la]$. For $u\in \lexab$, set
\[
f(u)=\begin{cases}
f_1(u), & u < x,
\\
y, & u=x,
\\
f_2(u), & u > x.
\end{cases}
\]
Then $f\in\oip{\lexab}$ and $f(x)=y$.
\end{proof}

\begin{proposition}\label{p:blexhome:1}
Let $\la$ be a regular cardinal and $\ka\leq\la$ a cardinal.
\begin{enumerate}
\item
If $x,y\in \homg l,\ka\la$ or $x,y\in \homg r,\ka\la$, then there is $f\in\oip{\lexab}$ with $f(x)=y$.
\item
If $x\in \homg l,\ka\la$, $y\in \homg r,\ka\la$, or $x\in \homg r,\ka\la$, $y\in \homg l,\ka\la$, then there is $f\in\oim{\lexab}$ with $f(x)=y$.
\item
If $x,y\in \homg lr,\ka\la$, then there is $f\in\oipm{\lexab}$ with $f(x)=y$ and $f(\homg lr,\ka\la)=\homg lr,\ka\la$.
\item
if $f\in\oip{\lexab}$, then $f(\homg l,\ka\la)=\homg l,\ka\la$, $f(\homg r,\ka\la)=\homg r,\ka\la$, and $f(\homg lr,\ka\la)=\homg lr,\ka\la$;
\item
$R_\la(\homg l,\ka\la)=\homg r,\ka\la$, $R_\la(\homg r,\ka\la)=\homg l,\ka\la$, and $R_\la(\homg lr,\ka\la)=\homg lr,\ka\la$;
\item
if $f\in\oim{\lexab}$, then $f(\homg l,\ka\la)=\homg r,\ka\la$, $f(\homg r,\ka\la)=\homg l,\ka\la$, and $f(\homg lr,\ka\la)=\homg lr,\ka\la$.
\end{enumerate}
\end{proposition}
\begin{proof}
Items (4), (5) and (6) are obvious. Item (1) follows from Proposition \ref{p:blexhome:1-1}.

Let us prove (2). Let $x'=R_\la(x)$. By (1), there is $f'\in\oip \lexab$ with $f'(x')=y$. Let $f = f' \circ R_\la$. Then $f\in\oim \lexab$ and $f(x)=y$.

Let us prove (3). If $x,y\in \homg l,\ka\la$ or $x,y\in \homg r,\ka\la$, then (3) follows from (1) and (4).
If $x\in \homg l,\ka\la$, $y\in \homg r,\ka\la$, or $x\in \homg r,\ka\la$, $y\in \homg l,\ka\la$, then (3) follows from (2) and (6).
\end{proof}

The following proposition follows from Proposition \ref{p:blexhome:1}.

\begin{proposition}\label{p:blexhome:2}
Let $\la$ be a regular cardinal.
\begin{enumerate}
\item
The family
\begin{gather*}
\{ \homg l,0\la, \homg r,0\la, \homg l,1\la, \homg r,1\la, \homm\la \} \cup 
\\
\set{\homg l,\ka\la, \homg r,\ka\la: \ka<\la\text{ is a regular cardinal}}
\end{gather*}
is the family of orbits of the action of the group $\oip\lexab$ on $\lexab$.
\item
The family
\begin{gather*}
\{ \homg lr,0\la,  \homg lr,1\la, \homm\la \} \cup 
\\
\set{\homg lr,\ka\la: \ka<\la\text{ is a regular cardinal}}
\end{gather*}
is the family of orbits of the action of the group $\oipm\lexab$ on $\lexab$.
\end{enumerate}
\end{proposition}

\begin{theorem}\label{t:blexhome:1}
Let $\ka\leq\la$ be regular cardinals. The spaces $\homg l,\ka\la$, $\homg r,\ka\la$ are order-homogeneous LOTS. The space $\homg lr,\ka\la$ is a $\pm$-order-homogeneous LOTS.
\end{theorem}
\begin{proof}
By Proposition \ref{p:blex:1+1}(3)(b), the spaces $\homg l,\ka\la$, $\homg r,\ka\la$ and $\homg lr,\ka\la$ are LOTS with the order inherited from $\lexab$. By Proposition \ref{p:blexhome:2}(1), the LOTS $\homg l,\ka\la$ and $\homg r,\ka\la$ are order-homogeneous. By Proposition \ref{p:blexhome:2}(2), the LOTS $\homg lr,\ka\la$ is $\pm$-order-homogeneous.
\end{proof}

\begin{cor}\label{c:blexhome:1}
Let $\ka\leq\la$ be regular cardinals. The spaces $\homg l,\ka\la$, $\homg r,\ka\la$ and $\homg lr,\ka\la$ are homogeneous LOTS.
\end{cor}


\section{Homogeneous extensions of LOTS} \label{sec:exos}
For a LOTS $X$ denote 
\begin{align*}
\er(X) &= \set{x\in X: x\text{ is the minimum or a right jump point of }X},
\\
\el(X) &= \set{x\in X: x\text{ is the maximum or a left jump point of }X},
\\
\ea(X) &=  \el(X)\cup\er(X).
\end{align*}

Let $L$ and $R$ be LOTS such that $L$ has a minimum and no maximum, and $R$ has a maximum and no minimum.
We distinguish every inserted point from the original points and record the side of its insertion. Thus $(x,t)$ abbreviates $(x,\mathrm{L},t)$ for a point inserted on the left and $(x,\mathrm{R},t)$ for one inserted on the right. These copies are disjoint even when $L$ and $R$ have the same underlying set. We use the same labels in partial and full extensions.
Denote
\begin{align*}
\cD(X,L,R) &= (\el(X) \times R) \sqcup (\er(X) \times L),
\\
\ex(X,L,R) &= X \sqcup \cD(X,L,R),
\\
\ex_0(X,L,R) &= X,
\\
\ex_n(X,L,R) &= \ex(\ex_{n-1}(X,L,R),L,R),
\\
\ex_\om(X,L,R) &= \bigcup_{n\in\om} \ex_n(X,L,R).
\end{align*}

We introduce a linear order on the set $\ex(X,L,R)$ as an ordered sum over the points of $X$. For $x\in X$ set
\[
L_x=\begin{cases}\sset x\times L,& x\in\er(X),\\ \es,& x\notin\er(X),\end{cases}
\qquad
R_x=\begin{cases}\sset x\times R,& x\in\el(X),\\ \es,& x\notin\el(X).\end{cases}
\]
The block at $x$ is the ordered sum $B_x=L_x+\sset x+R_x$, with its nonempty copies ordered as $L$ and $R$, respectively. If $x<x'$, every point of $B_x$ is less than every point of $B_{x'}$. This order extends the order of $X$. In particular, if $x\in\el(X)\cap\er(X)$, then $x$ is isolated in $X$ and its block is a copy of $L+\sset x+R$.


\begin{proposition}\label{p:exos:1}
Let $X$ be a compact LOTS and let $L$ and $R$ be LOTS in which every closed interval is compact. Suppose $L$ has a minimum and no maximum, and $R$ has a maximum and no minimum. Then the natural inclusion of $X$ into $\ex(X,L,R)$ is a topological embedding and $\ex(X,L,R)$ is compact. Each $\ex_n(X,L,R)$ is a compact subspace of $\ex_\om(X,L,R)$, so $\ex_\om(X,L,R)$ is $\sigma$-compact.
\end{proposition}
\begin{proof}
The assertion is immediate if $X=\es$. Otherwise, put $Y=\ex(X,L,R)$. Rays of $Y$ with endpoint in $X$ induce the corresponding open rays of $X$. If $u\in L_x$, the two open rays with endpoint $u$ have traces $\ilpp(-\infty,x)$ and $\ilcp[x,+\infty)$ on $X$. The latter is open because $x$ is the minimum or has an immediate predecessor in $X$. For $u\in R_x$ the reverse-order argument applies. This proves that $X$ has its original topology in $Y$.

Every block $B_x=L_x+\sset x+R_x$ has endpoints and no gaps. Indeed, a gap inside $L$ or $R$ would give a gap in a closed interval of that space, and the point $x$ supplies the missing endpoint of each nonempty copy adjacent to it. Thus $B_x$ is compact by Proposition \ref{p:1+4}. To see that $Y$ is order complete, let $A\subset Y$ be nonempty and let $s$ be the supremum in $X$ of the indices of the blocks meeting $A$. If $A$ meets $B_s$, then $\sup A=\sup(A\cap B_s)$ in $B_s$. Otherwise $\sup A=\min B_s$. The first and last blocks of $Y$ supply its endpoints, so $Y$ is compact.

Write $Y_n=\ex_n(X,L,R)$ and $Y_\om=\ex_\om(X,L,R)$. Iteration shows that each $Y_n$ is compact and embeds in every later finite stage. For a ray in $Y_\om$ and a fixed $n$, choose $m\ge n$ with the endpoint of the ray in $Y_m$. Its trace on $Y_n$ is open because it is the trace of an open ray of $Y_m$. Conversely, each open ray of $Y_n$ is induced by the ray of $Y_\om$ with the same endpoint. Hence $Y_n$ embeds in $Y_\om$, and $Y_\om=\bigcup_{n\in\om}Y_n$ is $\sigma$-compact.
\end{proof}

Below we shall omit the parameters $L$ and $R$: we shall write $\cD(X)$ instead of $\cD(X,L,R)$, and similarly for the other notation.

For $x\in \ea(X)$ and $M\subset \ea(X)$ set
\begin{align*}
\cD_{x}(X) &= L_x\sqcup R_x,
\\
\cD_{M}(X) &= \bigcup_{x\in M} \cD_{x}(X),
&
\ex(X;M) &= X \sqcup \cD_{M}(X).
\end{align*}

\begin{statement}\label{s:exos:1}
Let $M\subset \ea(X)$.
Then $\ex_\om(\ex(X;M))=\ex_\om(X)$.
\end{statement}
\begin{proof}
We first compare one partial extension with the full extension. Let $W$ be a LOTS, $A\subset\ea(W)$, and put $D=\cD_A(W)$, $Z=\ex(W;A)$ and $W'=\ex(W)$. For $s\in\{l,r\}$ we have
\[
\cB_s(Z)=\bigl(\cB_s(W)\setminus A\bigr)\cup\bigl(D\cap\cB_s(W')\bigr).
\]
Indeed, an original point in $A$ loses each boundary side on which a copy is inserted: $L$ has no maximum and $R$ has no minimum. At an original point where no left insertion is made, membership in $\er$ is unchanged. A minimum remains a minimum; if the point had an immediate predecessor, the block at that predecessor still has a greatest element. If the point had points below it but no immediate predecessor, no block at a smaller point can create one. The argument on the right is symmetric. This proves the identity on $W$.

For a point of $D$, its side types in $Z$ and $W'$ are determined within its inserted copy, except on the left at the minimum of a copy of $L$, or on the right at the maximum of a copy of $R$. Consider the former case, with the copy inserted at $w\in\er(W)$. If $w=\min W$, the minimum of this copy is the minimum in both extensions. Otherwise $w$ has an immediate predecessor in $W$, and the block at that predecessor has a greatest element in both extensions. The minimum of the copy therefore belongs to $\er$ in both. Reversing the order gives the assertion for a copy of $R$, proving the identity on $D$ as well.

It follows that passing from $Z$ to $\ex(Z)$ makes precisely the missing insertions at the original points of $W$ and, at points of $D$, the same insertions as in the extension of $W'$. With the ordering of the blocks and the labels fixed above, this gives
\[
\ex(Z)=\ex\bigl(W';D\cap\ea(W')\bigr).
\]
Now let $X_n=\ex_n(X)$, set $M_0=M$, and define
\[
M_{n+1}=\cD_{M_n}(X_n)\cap\ea(X_{n+1}),\qquad Z_n=\ex(X_n;M_n)
\quad(n\in\om).
\]
Applying the one-step identity with $W=X_n$ and $A=M_n$ yields $\ex(Z_n)=Z_{n+1}$. Since $X_n\subset Z_n\subset X_{n+1}$ for every $n$, taking unions gives
\[
\ex_\om(\ex(X;M))=\bigcup_{n\in\om}Z_n=\bigcup_{n\in\om}X_n=\ex_\om(X).
\]
\end{proof}

\begin{statement}\label{s:exos:2}
If $R$ is order reversing isomorphic to $L$, then every $f\in\oipm X$ extends to some $\hat f\in \oipm{\ex_\om(X)}$.
\end{statement}
\begin{proof}
Let $\ph: L\to R$ and $\psi: R\to L$ be order reversing isomorphisms such that $\ph\circ \psi$ is the identity map of $R$ onto itself.
\begin{claim}\label{cl:exos:1-1}
Let $Y$ be a LOTS. Every $f\in\oipm Y$ extends to some $\hat f\in \oipm{\ex(Y)}$.
\end{claim}
\begin{proof}
Suppose $f\in \oip Y$. Set
\[
\hat f(x) = \begin{cases}
f(x), & x\in Y,
\\
(f(y),z), & x=(y,z)\in \cD(Y).
\end{cases}
\]
Then $\hat f\in \oip{\ex(Y)}$.

Suppose $f\in \oim Y$. Set
\[
\hat f(x) = \begin{cases}
f(x), & x\in Y,
\\
(f(y),\ph(z)), & x=(y,z)\in \er(Y) \times L,
\\
(f(y),\psi(z)), & x=(y,z)\in \el(Y) \times R.
\end{cases}
\]
Then $\hat f\in \oim{\ex(Y)}$.
\end{proof}
By Fact \ref{cl:exos:1-1}, there exist $f_n\in\oipm{\ex_n(X)}$ for $\nom$ such that $f_0=f$ and $f_n=\restr{f_{n+1}}{\ex_n(X)}$ for $\nom$. Define $\hat f: \ex_\om(X)\to \ex_\om(X)$ by $\hat f(x)=f(x)$ for $x\in X$, and $\hat f(x)=f_n(x)$ for $x\in \ex_n(X)\setminus \ex_{n-1}(X)$.
\end{proof}

\begin{theorem}\label{t:exos:1}
Let $X$, $L$ and $R$ be LOTS.
Let $L$ have a minimum and no maximum, and let $L$ and $\rev R$ be order isomorphic.
Suppose the following conditions hold:
\begin{enumerate}
\item
if $x,y\in X\setminus \ea(X)$, then $f(x)=y$ for some $f\in\oipm X$;
\item
the LOTS $\ilpc(-\infty,a] + R + \ilpp(a,+\infty)$ and $\ilpp(-\infty,b) + L + \ilcp[b,+\infty)$ are order isomorphic to $X$ for $a\in \el(X)$ and $b\in \er(X)$.
\end{enumerate}
Then $\ex_\om(X,L,R)$ is a $\pm$-order-homogeneous LOTS.
\end{theorem}
\begin{proof}
Let us show that condition (2) implies that $\ex(X;\sset x)$ is order isomorphic to $X$ for every $x\in\ea(X)$. If $x$ belongs to only one of the sets $\el(X)$, $\er(X)$, this is exactly condition (2). Let $x\in\el(X)\cap\er(X)$. Then $\ex(X;\sset x)$ is obtained from $X$ by two successive insertions at the point $x$. Let $X_1=\ilpc(-\infty,x]+R+\ilpp(x,+\infty)$; by (2) there is an order isomorphism $\varphi: X_1\to X$. The neighborhood of $x$ on the left is the same in $X$ and $X_1$, so $x$ is the minimum or a right jump point of $X_1$, and hence $\varphi(x)\in\er(X)$. By (2), the LOTS $\ilpp(-\infty,\varphi(x))+L+\ilcp[\varphi(x),+\infty)$ is order isomorphic to $X$. Transferring this insertion through $\varphi^{-1}$, we get that the LOTS $\ilpp(-\infty,x)+L+\sset x+R+\ilpp(x,+\infty)=\ex(X;\sset x)$ is order isomorphic to $X_1$ and, hence, to $X$. (Note that in the application to the space $\sS$ there are no isolated points, and this remark is not needed.)
\begin{claim}\label{cl:exos:1}
Let $u\in \ex_\om(X)$. There is $Y\subset \ex_\om(X)$ such that $X\subset Y$, $\ex_\om(Y)=\ex_\om(X)$, $Y$ is order isomorphic to $X$, and $u\in Y\setminus \ea(Y)$.
\end{claim}
\begin{proof}
For some $n\in\om$,
\begin{align*}
u &= (x,u_1,...,u_n)\in \ex_n(X).
\end{align*}
Set $x_0=x$, $x_1=(x,u_1)$, ..., $x_n=(x,u_1,...,u_n)=u$. Define $X_i\subset \ex_\om(X)$ for $i=0,1,...,n$ such that $x_i\in X_i$ for $i\leq n$ and $x_i \in \ea(X_i)$ for $i<n$. Set $X_0=X$ and $X_i=\ex(X_{i-1};\sset{x_{i-1}})$ for $i>0$. By Statement \ref{s:exos:1}, $\ex_\om(X_i) = \ex_\om(X)$, and by condition (2), $X_i$ is order isomorphic to $X$ for $i=0,1,...,n$. If $u=x_n\notin \ea(X_n)$, set $Y=X_n$. If $u\in \ea(X_n)$, set $Y=\ex(X_{n};\sset{u})$. Then $u\notin \ea(Y)$. By Statement \ref{s:exos:1} and (2), $Y$ is order isomorphic to $X$ and $\ex_\om(Y)=\ex_\om(X)$.
\end{proof}
Let $x,y\in\ex_\om(X)$. By Fact \ref{cl:exos:1}, there is $Y'\subset \ex_\om(X)$ such that $X\subset Y'$, $x\in Y'\setminus \ea(Y')$, $\ex_\om(Y')=\ex_\om(X)$, and $Y'$ is order isomorphic to $X$. By Fact \ref{cl:exos:1}, there is $Y\subset \ex_\om(Y')$ such that $Y'\subset Y$, $y\in Y\setminus \ea(Y)$, $\ex_\om(Y)=\ex_\om(Y')=\ex_\om(X)$, and $Y$ is order isomorphic to $Y'$ and to $X$.

Thus $x,y\in Y\setminus \ea(Y)$, $\ex_\om(Y)=\ex_\om(X)$, and $Y$ is order isomorphic to $X$.
By condition (1), $f(x)=y$ for some $f\in\oipm Y$. By Statement \ref{s:exos:2}, $f$ extends to some $\hat f\in \oipm{\ex_\om(Y)}$. Then $\hat f(x)=y$ for $\hat f\in \oipm{\ex_\om(X)}$. Hence the LOTS $\ex_\om(X)$ is $\pm$-order-homogeneous.
\end{proof}


\section{Large compacta in homogeneous LOTS} \label{sec:bcs}

Set $\sS=(\lex{\om_1}\B\setminus \homm{\om_1})\cup (\homm{\om_1}\times\B)$. We define a linear order on $\sS$. The order on $\lex{\om_1}\B\setminus \homm{\om_1}$ is inherited from $\lex{\om_1}\B$, and the order on $\homm{\om_1}\times\B$ is lexicographic. Let $x\in \lex{\om_1}\B\setminus \homm{\om_1}$ and $(y,a)\in \homm{\om_1}\times\B$. Then $x<(y,a)$ if and only if $x<y$, and $x>(y,a)$ if and only if $x>y$.

Define the map $\pi: \sS\to \lex{\om_1}\B$ by $\pi(x)=x$ for $x\in \lex{\om_1}\B\setminus \homm{\om_1}$, and $\pi((y,a))=y$ for $(y,a)\in \homm{\om_1}\times\B$. The map $\pi$ is a continuous increasing map of compact LOTS.

Define the order reversing automorphism $\Rs: \sS\to\sS$ by $\Rs(x)=R_{\om_1}(x)$ for $x\in \lex{\om_1}\B\setminus \homm{\om_1}$, $\Rs((y,0))=(R_{\om_1}(y),1)$, and $\Rs((y,1))=(R_{\om_1}(y),0)$ for $y\in  \homm{\om_1}$. Note that $\pi\circ \Rs=R_{\om_1}\circ \pi$.

For cardinals $\nu$ and $\mu$ set
\begin{align*}
\sS_{\nu,\mu} &= \set{x\in\sS: \cf_l(x,\sS)=\nu,\ \cf_r(x,\sS)=\mu},
&
\wS_{\nu,\mu} &= \sS_{\nu,\mu} \cup \sS_{\mu,\nu}.
\end{align*}

Note that $\mns$ is the minimum of $\sS$, and $\mxs$ is the maximum of $\sS$. From the construction of the space $\sS$ and Propositions \ref{p:blex:1+1} and \ref{p:1+1} we get
\begin{align*}
\sS_{0,\om_1}&=\sset{\mns},
\qquad
\sS_{\om_1,0}=\sset{\mxs},
\qquad
\wS_{0,\om_1}=\sset{\mns,\mxs},
\\
\sS_{1,\om_1}&=\set{x\in\sS : x\text{ is a right jump point}},
\\
\sS_{\om_1,1}&=\set{x\in\sS : x\text{ is a left jump point}},
\\
\wS_{1,\om_1}&=\set{x\in\sS : x\text{ is a jump point}},
\\
\wS_{\om,\om_1}&=\set{x\in\sS : \pi\chi(\pi(x),\lex{\om_1}\B)=\om}
\end{align*}
and that the families
\begin{align*}
\cS &=
\{\,
\sS_{0,\om_1}, \sS_{\om_1,0}, \sS_{\om,\om_1}, \sS_{\om_1,\om}, \sS_{1,\om_1}, \sS_{\om_1,1}
\,\},
\\
\wcS &=
\{\,
\wS_{0,\om_1}, \wS_{\om,\om_1}, \wS_{1,\om_1}
\,\},
\end{align*}
are partitions of the set $\sS$.
We now compute the character, weight and cellularity of $\sS$.

\begin{proposition}\label{p:bsc:1-1}
$\chi(x,\sS)=\om_1$ for all $x\in\sS$, $w(\sS)=|\sS|=2^{\om_1}$, and $c(\sS)=2^\om$.
\end{proposition}
\begin{proof}
The assertion about character follows from the partition above and Proposition \ref{p:1+1.5}. Write $B=\lex{\om_1}\B$. Since $\pi:\sS\to B$ is onto and its fibers have at most two points, $|\sS|=2^{\om_1}$. There are $2^{\om_1}$ points in $\homm{\om_1}$: partition $\om_1$ into three sets of cardinality $\om_1$, prescribe zeros on the first and ones on the second, and vary the digits freely on the third. Both digits then occur cofinally often, as required by Proposition \ref{p:blex:1}(1). Each such point yields a distinct jump in $\sS$. Distinct jumps require distinct basic open sets, as in the proof of Proposition \ref{p:blex:3-1}, so $2^{\om_1}\le w(\sS)\le|\sS|$.

Every nonempty open subset of $\sS$ contains a nonempty open convex set $I$. The set $I$ is infinite since $\sS$ has no isolated points. As the fibers of $\pi$ have size at most two, we can choose $u<v$ in $I$ so that $V=\ilpp({\pi(u)},{\pi(v)})_B$ is nonempty. Monotonicity gives $\pi^{-1}(V)\subset I$. Thus any disjoint family of nonempty open sets in $\sS$ yields a disjoint family of the same size in $B$. Conversely, such a family in $B$ lifts to one in $\sS$ by continuity and surjectivity of $\pi$. It follows from Proposition \ref{p:blex:3} that $c(\sS)=c(B)=2^\om$.
\end{proof}

\begin{proposition}\label{p:bsc:1}
Let $a,b,c,d\in \sS$, $a<b$ and $c<d$.
The following conditions are equivalent:
\begin{enumerate}
\item
the LOTS $\ilcc[a,b]$ and $\ilcc[c,d]$ are order isomorphic and have no isolated points;
\item
$\cf_r(a,\sS)=\cf_r(c,\sS)\geq \om$ and
$\cf_l(b,\sS)=\cf_l(d,\sS)\geq \om$.
\end{enumerate}
\end{proposition}
\begin{proof}
An order isomorphism of the intervals preserves the right cofinality at their minima and the left cofinality at their maxima. If neither interval has isolated points, these cofinalities are infinite. This proves (1) $\rarr$ (2).

Assume (2) and write $B=\lex{\om_1}\B$. Each fiber of $\pi$ is a singleton or a pair of consecutive points. Since $\cf_r(a,\sS)$ is infinite, $a$ is the greatest point of its fiber. Similarly, $b$ is the least point of its fiber, and the same holds for $c$ and $d$, respectively. In particular, $\pi(a)<\pi(b)$ and $\pi(c)<\pi(d)$. The points above $a$ project onto the points above $\pi(a)$. Projection sends coinitial sets to coinitial sets, and choosing the least point of each fiber lifts any coinitial set. Hence $\cf_r(\pi(a),B)=\cf_r(a,\sS)$. The same argument at $c$ and its reverse at $b,d$ show that the projected intervals satisfy condition (2) of Theorem \ref{t:blex:3}. Let $f':\ilcc[\pi(a),\pi(b)]\to\ilcc[\pi(c),\pi(d)]$ be the resulting order isomorphism.

At an interior point of either projected interval, both one-sided cofinalities are the same as in $B$. Thus $f'$ preserves the interior points of type $(\om_1,\om_1)$, precisely the points of $\homm{\om_1}$ whose fibers have two elements. Lift $f'$ over each interior point by the unique order isomorphism between its fiber and the corresponding fiber in the target. At the endpoints, the fibers in the restricted intervals are singletons, so set $f(a)=c$ and $f(b)=d$. These maps together give an order isomorphism $f:\ilcc[a,b]\to\ilcc[c,d]$. Interior points are non-isolated since $\sS$ has no isolated points, and condition (2) ensures that the endpoints are non-isolated as well.
\end{proof}

\begin{proposition}\label{p:bsc:1+1}
If $(a,b)\in \jumps\sS$, then $\ilcc[\mns,a]$ and $\ilcc[b,\mxs]$ are order isomorphic to $\sS$.
\end{proposition}
\begin{proof}
Since $\cf_r(\mns,\sS)=\cf_l(a,\sS)=\cf_r(b,\sS)=\cf_l(\mxs,\sS)=\om_1$, Proposition \ref{p:bsc:1} implies that the LOTS $\ilcc[\mns,a]$, $\ilcc[b,\mxs]$ and $\sS$ are order isomorphic.
\end{proof}

\begin{proposition}\label{p:bsc:2}
Let $x,y\in \sS$.
The following conditions are equivalent:
\begin{enumerate}
\item
$f(x)=y$ for some $f\in\oip\sS$;
\item
$\cf_r(x,\sS)=\cf_r(y,\sS)$ and
$\cf_l(x,\sS)=\cf_l(y,\sS)$.
\end{enumerate}
\end{proposition}
\begin{proof}
Let
$\al=\cf_l(x,\sS)=\cf_l(y,\sS)$ and
$\be = \cf_r(x,\sS)=\cf_r(y,\sS)$.
Clearly, (1) $\rarr$ (2). Let us prove (2) $\rarr$ (1).
If $0\in\sset{\al,\be}$, then $x=y$ and the statement is obvious.

Consider the case $1\in\sset{\al,\be}$. Consider the case $\al=1$. Then $x$ and $y$ are right jump points and $\be=\om_1$. Let $x'$ and $y'$ be the corresponding left jump points. By Proposition \ref{p:bsc:1+1}, there exist order isomorphisms $f_1: \ilcc[\mns,x']\to \ilcc[\mns,y']$ and $f_2: \ilcc[x,\mxs]\to \ilcc[y,\mxs]$. For $u\in \sS$, set
\[
f(u)=\begin{cases}
f_1(u), & u \leq x',
\\
f_2(u), & u \geq x.
\end{cases}
\]
Then $f\in\oip{\sS}$ and $f(x)=y$.
Consider the case $\be=1$. Then $\cf_l(\Rs(x),\sS)=\cf_l(\Rs(y),\sS)=1$ and $\cf_r(\Rs(x),\sS)=\cf_r(\Rs(y),\sS)=\om_1$. Hence $f'(\Rs(x))=\Rs(y)$ for some $f'\in \oip{\sS}$. We obtain $f(x)=y$ for $f=\Rs\circ f' \circ \Rs\in \oip\sS$.

Consider the case when $\al$ and $\be$ are infinite cardinals.
Since $\mns<x<\mxs$, $\mns<y<\mxs$, $\al\geq \om$, and $\be\geq \om$, Proposition \ref{p:bsc:1} gives order isomorphisms $f_1: \ilcc[\mns,x]\to \ilcc[\mns,y]$ and $f_2: \ilcc[x,\mxs]\to \ilcc[y,\mxs]$. For $u\in \sS$, set
\[
f(u)=\begin{cases}
f_1(u), & u < x,
\\
y, & u=x,
\\
f_2(u), & u > x.
\end{cases}
\]
Then $f\in\oip{\sS}$ and $f(x)=y$.
\end{proof}

\begin{proposition}\label{p:bsc:3}
\begin{enumerate}
\item
The family $\cS$
is the family of orbits of the action of the group $\oip\sS$ on $\sS$.
\item
The family $\wcS$ is the family of orbits of the action of the group $\oipm\sS$ on $\sS$.
\end{enumerate}
\end{proposition}
\begin{proof}
Since $\cS$ is a partition of the space $\sS$, item (1) follows from Proposition \ref{p:bsc:2}.
Item (2) follows from the fact that $\Rs\in\oim\sS$, $\wcS=\set{M\cup \Rs(M): M\in\cS}$, and item (1).
\end{proof}

\begin{proposition}\label{p:bsc:4}
The following LOTS are order isomorphic to $\sS$: $\sS + \sS$, $\om \times \sS + \sS$, $\sS+ \rev (\om \times \sS)$.
\end{proposition}
\begin{proof}
Let $(a,b)\in \jumps \sS$. By Proposition \ref{p:bsc:1+1}, the LOTS $\ilcc[\mns,a]$, $\ilcc[b,\mxs]$ and $\sS$ are order isomorphic. Hence $\sS$ and $\sS+\sS$ are order isomorphic.

Take a sequence of jumps $(a_i,b_i)_{i\in\om}\subset \jumps \sS$ such that $b_i<a_{i+1}$ for $i\in\om$. Let $b_{-1}=\mns$ and $b=\sup \set{b_i:i\in\om}$. Then $\cf_r(b_{i-1},\sS)=\cf_l(a_i,\sS)=\om_1$ for $i\in\om$, and $\cf_r(b,\sS)=\cf_l(\mxs,\sS)=\om_1$. By Proposition \ref{p:bsc:1}, the LOTS $\ilcc[b_{i-1},a_i]$ for $i\in\om$ and $\ilcc[b,\mxs]$ are order isomorphic to $\sS$. Hence $\ilcp[\mns,b)$ is order isomorphic to $\om\times \sS$. Since $\sS=\ilcp[\mns,b)+\ilcc[b,\mxs]$, $\sS$ is order isomorphic to $\om\times \sS + \sS$.

Since $\sS$ is order isomorphic to $\rev \sS$, $\sS+ \rev(\om\times \sS)$ is order isomorphic to $\rev(\om\times \sS + \sS)$. Hence $\sS+ \rev(\om\times \sS)$ is order isomorphic to $\sS$.
\end{proof}

\begin{theorem}\label{t:bsc:1}
The LOTS $\ex_\om(\sS, \om\times \sS, \rev(\om\times \sS))$ is $\pm$-order-homogeneous.
\end{theorem}
\begin{proof}
Let us check that the conditions of Theorem \ref{t:exos:1} hold. Since $\sS\setminus \ea(\sS)=\wS_{\om,\om_1}$, condition (1) of Theorem \ref{t:exos:1} follows from Proposition \ref{p:bsc:3}(2).

Let us check condition (2). Let $L=\om\times \sS$ and $R=\rev(\om\times \sS)$, $a\in \el(\sS)$ and $b\in \er(\sS)$.

If $a=\max \sS$, then $A=\ilpc(-\infty,a] + R + \ilpp(a,+\infty)=\sS + \rev(\om\times \sS)$. By Proposition \ref{p:bsc:4}, $A$ is order isomorphic to $\sS$. If $a$ is a left jump point, then $A= \ilcc[\mns,a] + \rev(\om\times \sS) + \ilcc[a',\mxs]$, where $a'=\min \ilpp(a,+\infty)$. By Proposition \ref{p:bsc:1+1}, $A$ is order isomorphic to $\sS + \rev(\om\times \sS) + \sS$. By Proposition \ref{p:bsc:4}, $\sS + \rev(\om\times \sS)$ is order isomorphic to $\sS$. Hence $A$ is order isomorphic to $\sS + \sS$. By Proposition \ref{p:bsc:4}, $A$ is order isomorphic to $\sS$.

If $b=\min \sS$, then $B=\ilpp(-\infty,b) + L + \ilcp[b,+\infty)=\om\times \sS + \sS$. By Proposition \ref{p:bsc:4}, $B$ is order isomorphic to $\sS$. If $b$ is a right jump point, then $B= \ilcc[\mns,b'] + \om\times \sS + \ilcc[b,\mxs]$, where $b'=\max \ilpp(-\infty,b)$. By Proposition \ref{p:bsc:1+1}, $B$ is order isomorphic to $\sS + \om\times \sS + \sS$. By Proposition \ref{p:bsc:4}, $\om\times \sS+\sS$ is order isomorphic to $\sS$. Hence $B$ is order isomorphic to $\sS + \sS$. By Proposition \ref{p:bsc:4}, $B$ is order isomorphic to $\sS$.

By Theorem \ref{t:exos:1}, $\ex_\om(\sS, \om\times \sS, \rev(\om\times \sS))$ is $\pm$-order-homogeneous.
\end{proof}


\section{Examples and questions} \label{sec:qe}

A semitopological group $G$ that is a LOTS and not a $P$-space is metrizable (Theorem \ref{t:main:3}). In this statement the condition that $G$ is not a $P$-space cannot be dropped, and the condition that $G$ is a LOTS cannot be weakened to $G$ being a GO space.

The following statement follows from Proposition \ref{p:h:tg1}.
\begin{example}\label{e:1}
The group $\lex{\om_1}\R$ with the order topology is a non-metrizable LOTS group that is a $P$-space.
\end{example}

\begin{example}\label{e:2}
The Sorgenfrey line is a non-metrizable, first-countable GO paratopological group.
\end{example}

A semitopological group $G$ that is a GO space and not a $P$-space is submetrizable, paracompact, and has countable character (Theorem \ref{t:main:2}). In this statement the condition that $G$ is a semitopological group cannot be replaced by the condition that $G$ is a homogeneous space.

For a regular cardinal $\la$ denote $X_\la=\homg lr,\om\la$.

\begin{example}\label{e:3}
Let $\la$ be an uncountable regular cardinal.
The space $X_\la$ is a homogeneous countably compact LOTS with $\chi(X_\la)=\la$ and $\pi\chi(X_\la)=\om$.
\end{example}
\begin{proof}
By Corollary \ref{c:blexhome:1}, the space $X_\la$ is homogeneous. By Proposition \ref{p:blex:1+1}, $X_\la$ is a countably compact LOTS with $\pi\chi(X_\la)=\om$ and $\chi(X_\la)=\la$.
\end{proof}

By Theorem \ref{t:main:1}, if $\la>\om_1$ is a regular cardinal, then compact subsets of $X_\la$ are first countable.

\begin{example}\label{e:4}
The space $X_{\om_1}$ is a homogeneous countably compact LOTS with $\chi(X_{\om_1})=\om_1$, $\pi\chi(X_{\om_1})=\om$, and the transfinite line $\om_{1}+1$ embeds in $X_{\om_1}$.
\end{example}
\begin{proof}
By Proposition \ref{p:blex:1+1}, the transfinite line $\om_{1}+1$ embeds in $X_{\om_1}$.
\end{proof}

The LOTS $\lex {\om_1} \B$ does not embed in a homogeneous GO space (Proposition \ref{p:blex:5}).
Let $L=\om\times\sS$, $R=\rev L$, $\Hs=\ex_\om(\sS,L,R)$ and $X=\ex(\sS,L,R)$.

\begin{example}\label{e:5}
The space $\Hs$ is a homogeneous $\sigma$-compact LOTS with $\chi(\Hs)=\om_1$ and $\pi\chi(\Hs)=\om$. The first extension $X$ of $\sS$ is a compact subspace of $\Hs$ satisfying
\[
\chi(x,X)=\om_1 \quad (x\in X), \qquad c(X)=w(X)=|X|=2^{\om_1}.
\]
Moreover, $c(\Hs)=2^{\om_1}$.
\end{example}
\begin{proof}
Theorem \ref{t:bsc:1} gives the homogeneity of $\Hs$. Each closed interval of $L$ is a closed subspace of a finite ordered sum of compact copies of $\sS$, hence is compact. The same holds for $R$. By Proposition \ref{p:exos:1}, $\Hs$ is $\sigma$-compact and $\sS\subset X\subset\Hs$ are natural embeddings of the two compact stages. Since $\chi(x,\sS)=\om_1$ for every $x\in\sS$ by Proposition \ref{p:bsc:1-1}, Theorem \ref{t:main:1}(2) applied to $\sS\subset\Hs$ gives $\chi(\Hs)=\om_1$ and $\pi\chi(\Hs)=\om$.

Every point $x\in X$ belongs to a subspace $S_x$ homeomorphic to $\sS$. For an original point take $S_x=\sS$. An inserted point lies in a convex copy of $\sS$ or $\rev\sS$ inside one of the inserted copies of $L$ or $R$. Thus
\[
\om_1=\chi(x,S_x)\leq\chi(x,X)\leq\chi(\Hs)=\om_1.
\]

By the proof of Proposition \ref{p:bsc:1-1}, there are $2^{\om_1}$ points $z\in\homm{\om_1}$. Each yields consecutive points $a_z=(z,0)$ and $b_z=(z,1)$ in $\sS$. In $X$ the interval $\ilpp(a_z,b_z)$ contains the copy of $R$ inserted after $a_z$ and the copy of $L$ inserted before $b_z$. As $z$ varies, these intervals are nonempty, open and pairwise disjoint. The intervals with the same endpoints in $\Hs$ are also nonempty and pairwise disjoint. Hence $c(X)\geq 2^{\om_1}$ and $c(\Hs)\geq 2^{\om_1}$.

Finally, $|\sS|=|L|=|R|=2^{\om_1}$ by Proposition \ref{p:bsc:1-1}. Each finite stage adds at most two such copies per point of the preceding stage, so each stage and their countable union $\Hs$ have cardinality $2^{\om_1}$. Since $c(\Hs)\leq|\Hs|$ and $c(X)\leq w(X)\leq|X|$ (the intervals and rays of $X$ form a base), the stated equalities follow.
\end{proof}

\begin{question}\label{q:1-1}
Can the compact space $\sS$ be embedded in a homogeneous countably compact LOTS? Let $K$ be a nonempty compact subset of some homogeneous countably compact LOTS. Is it true that $\chi(x,K)\leq\om$ for some $x\in K$?
\end{question}

\begin{example}\label{e:6}
The \term{double arrow} is the compact LOTS
\[
D=\bigl([0,1]\times\B\bigr)\setminus\{(0,0),(1,1)\},
\]
with the lexicographic order \cite[p.~114]{HartKunen1999}. The two deleted points are the isolated endpoints of the full lexicographic product. Banakh \cite{Banakh2008} gives a compatible left-topological group operation on $D$; the opposite operation makes $D$ a right-topological group.
\end{example}

\begin{question}\label{q:1}
Does the LOTS $X_{\om_1}$ admit the structure of a right-topological group?
\end{question}

\begin{question}\label{q:1+1}
Does the LOTS $\Hs$ admit the structure of a right-topological group?
\end{question}

\begin{question}\label{q:2}
Let $X$ be a $p$-homogeneous GO (LOTS) space.
Is it true that every compact subset of $X$ has tightness at most $\om_1$?
\end{question}

\begin{question}\label{q:3}
Let $X$ be a monotonically normal ($p$-)homogeneous space.
Is it true that every compact subset of $X$ has tightness at most $\om_1$?
\end{question}

There are many examples of separable, non-first-countable, non-metrizable monotonically normal topological groups \cite{GartsideReznichenko2000nmfs}. Such groups are stratifiable spaces
\cite[Theorem 24]{GartsideReznichenko2000nmfs}.

For monotonically normal paratopological groups, and thus also for semitopological groups, none of separability, countable pseudocharacter, or the presence of a non-trivial convergent sequence suffices for stratifiability. Indeed, the Sorgenfrey line with addition (Example \ref{e:2}) is monotonically normal as a GO space and has all three properties: $\Q$ is dense, $\bigcap_{n\geq 1}[0,1/n)=\sset 0$, and $1/n\to 0$. Yet its clopen subspace $[0,1)$ is not even semi-stratifiable \cite[Example 3.2]{Banakh2002quarter}, so the line is not stratifiable. Its inversion $x\mapsto -x$ is discontinuous; in fact it admits no compatible quasitopological group structure \cite{Shen2014quasi}. This example therefore does not decide the quasitopological cases of Questions \ref{q:4}, \ref{q:6} and \ref{q:8}.

\begin{question}\label{q:4}
Let $G$ be a separable monotonically normal quasitopological group. Is $G$ stratifiable?
\end{question}

\begin{question}[{\cite[Problem 6]{Collins1996}}]\label{q:5}
Let $G$ be a monotonically normal topological group with countable pseudocharacter. Is it true that $G$ is a stratifiable space?
\end{question}

\begin{question}\label{q:6}
Let $G$ be a monotonically normal quasitopological group with countable pseudocharacter. Is $G$ stratifiable?
\end{question}

Monotonically normal topological vector spaces are stratifiable spaces \cite[Theorem 1]{Shkarin2004}.

\begin{question}[{\cite[Remarks (2)]{Shkarin2004}}]\label{q:7}
Let $G$ be a monotonically normal topological group and let $G$ contain a non-trivial convergent sequence (that is, a sequence that converges to a point not belonging to the set of its terms). Is it true that $G$ is a stratifiable space?
\end{question}

\begin{question}\label{q:8}
Let $G$ be a monotonically normal quasitopological group containing a non-trivial convergent sequence. Is $G$ stratifiable?
\end{question}

Monotonically normal paratopological groups are hereditarily paracompact \cite[Theorem 2.6]{BuzyakovaVural2014}, and hence compact subsets of such groups have countable tightness \cite[Theorem 3.5]{WilliamsZhou1998} (Theorem \ref{t:main:3+1}). Compact subsets of monotonically normal topological groups are metrizable \cite[Theorem 2a]{Gartside1998}.

\begin{question}\label{q:9}
Let $G$ be a monotonically normal paratopological (semitopological, quasitopological) group. Is it true that compact subsets of $G$ are metrizable (separable, first countable)?
\end{question}



\end{document}